\documentclass[12pt]{article}
\usepackage{amssymb,amsthm,hyperref}
\newtheorem{thm}{Theorem}[section]
 
 \newtheorem{cor}{Corollary}[section]
 \newtheorem{lem}{Lemma}[section]
 \newtheorem{prop}{Proposition}[section]
 \newtheorem{defn}{Definition}[section]%是否是全文计数？
\theoremstyle{remark}
\newtheorem{rem}{Remark}[section]

\title{Global Well-posedness in Critical Besov Spaces for Two-fluid Euler-Maxwell Equations}
\author{}

\author{Jiang Xu\thanks {E-mail: jiangxu\underline{ }79@yahoo.com.cn}\\
\small{\textit{Department of Mathematics}},
\\ \small{\textit{Nanjing
University of Aeronautics and Astronautics}},
 \small{\textit{Nanjing 210016, China}}\\[5mm]Jun Xiong\thanks {E-mail: junx\underline{ }87@yahoo.cn}\\
\small{\textit{Department of Mathematics}},\\\small{\textit{Nanjing
University of Aeronautics and Astronautics}},
\small{\textit{Nanjing 211106, China}}\\[5mm]
Shuichi Kawashima\thanks{E-mail: kawashim@math.kyushu-u.ac.jp}\\
\small{\textit{Graduate School of Mathematics}},\\
\small{\textit{Kyushu University, Fukuoka 812-8581, Japan}}}

\date{}
\begin{document}
\maketitle{} \begin{abstract} In this paper, we study the
well-posedness in critical Besov spaces for two-fluid Euler-Maxwell
equations, which is different from the one-fluid case. We need to
deal with the difficulties mainly caused by the nonlinear coupling
and cancelation between two carriers. Precisely, we first obtain the
local existence and blow-up criterion of classical solutions to the
Cauchy problem and periodic problem pertaining to data in Besov
spaces with critical regularity. Furthermore, we construct the
global existence of classical solutions with aid of a different
energy estimate (in comparison with the one-fluid case) provided the
initial data is small under certain norms. Finally, we establish the
large-time asymptotic behavior of global solutions near equilibrium
in Besov spaces with relatively lower regularity.
\end{abstract}

\hspace{-0.5cm}\textbf{Keywords.} \small{two-fluid Euler-Maxwell
equations, classical solutions, Chemin-Lerner spaces}\\

\hspace{-0.5cm}\textbf{AMS subject classification:} \small{35L45,\
76N15,\ 35B25}

\section{Introduction}
As an un-magnetized plasma is operated under some high frequency
conditions (such as photoconductive switches, electro-optics and
high-speed computers, etc.), electromagnetic fields are generated by
moving electrons and ions, then the two carriers transport interacts
with the propagating magnetic waves. In this case, the transport
process is typically governed by Euler-Maxwell equations, which take
the form of Euler equations for the conservation laws of mass
density and current density for carriers, coupled to Maxwell's
equations for self-consistent electromagnetic fields. By some
appropriate re-scaling, the two-fluid compressible Euler-Maxwell
equations are written, in nondimensional form, as (see, e.g.,
\cite{MRS})
\begin{equation}
\left\{
\begin{array}{l}
\partial_{t}n_{\pm}+\nabla\cdot(n_{\pm}u_{\pm})=0,\\
\partial_{t}(n_{\pm}u_{\pm})+\nabla\cdot(n_{\pm}u_{\pm}\otimes u_{\pm})+\nabla p_{\pm}(n_{\pm})
=\mp n_{\pm}(E+\varepsilon u_{\pm}\times B)-n_{\pm}u_{\pm}/\tau_{\pm},\\
\varepsilon\lambda^2\partial_{t}E-\nabla\times B=\varepsilon(n_{+}u_{+}-n_{-}u_{-}),\\
 \varepsilon\partial_{t}B+\nabla\times E=0,\\
 \lambda^2\nabla\cdot E=n_{-}-n_{+},\ \ \ \nabla\cdot B=0.
 \end{array}
\right.\label{R-E1}
\end{equation}
for $(t,x)\in[0,+\infty)\times\Omega(\Omega=\mathbb{R}^{N}$ or
$\mathbb{T}^{N},\ N=2,3)$. Here the unknowns
$n_{\pm}=n_{\pm}(t,x)>0, u_{\pm}=u_{\pm}(t,x)\in\Omega$,
respectively, stand for densities and velocities of the electrons
(+) and ions $(-)$. $E=E(t,x)\in\Omega$ and $B=B(t,x)\in\Omega$
denote the electric field and magnetic field, respectively. The
pressure functions $p_{\pm}(\cdot)$ satisfy the usual $\gamma$-law:
$p_{\pm}(n_{\pm})=A_{\pm}n_{\pm}^{\gamma}$, where $A_{\pm}>0$ are
some physical constants and the adiabatic exponent $\gamma\geq1$.
$\tau_{\pm}$ are the (scaled) constants for the momentum-relaxation
timed of electrons and ions, and $\lambda>0$ is the Debye length.
$c=(\epsilon_{0}\upsilon_{0})^{-\frac{1}{2}}>0$ is the speed of
light, where $\epsilon_{0}$ and $\upsilon_{0}$ are the vacuum
permittivity and permeability. Setting $\varepsilon=\frac{1}{c}$.
The parameters $\tau_{\pm},\lambda$ and $\varepsilon$ arising from
nondimensionalization are independent each other, and they are
assumed to be very small compared to the reference physical size. In
this paper, we set these physical constants to be one.

It is not difficult to see  below in the text that the Euler-Maxwell
equations (\ref{R-E1}) consist of a quasi-linear symmetrizable
hyperbolic system, the main feature of which is the finite time
blow-up of classical solutions even when the initial data are smooth
and small. Hence, the qualitative study and device simulation of
(\ref{R-E1}) are far from trivial. The primary objective of this
paper is to establish the global well-posedness for the
corresponding Cauchy problem and periodic problem. For this purpose,
(\ref{R-E1}) is equipped with the following initial data
\begin{equation}(n_{\pm},u_{\pm},E,B)(x,0)=(n_{\pm0},u_{\pm0},E_{0},B_{0})(x)\label{R-E2}
\end{equation}
satisfying the compatible conditions
\begin{equation}\nabla\cdot E_{0}=n_{-0}-n_{+0},\ \ \ \nabla\cdot B_{0}=0,\ \
 x\in\Omega.\label{R-E3}
\end{equation}

In the past years, the Euler-Maxwell equations have attached much
attention. In one space dimension, using the Godunov scheme with the
fractional step together with the compensated compactness theory,
Chen, Jerome and Wang \cite{CJW} constructed the existence of a
global weak solution to the initial boundary value problem for
arbitrarily large initial data in $L^\infty$. Assuming initial data
in Sobolev spaces $H^s(\mathbb{R}^3)$ with higher regularity
$s>5/2$, a local existence theory of smooth solutions for the Cauchy
problem of Euler-Maxwell equations was established in \cite{J} by
the author's modification of the classical semigroup-resolvent
approach of Kato \cite{K}. Subsequently, the global existence and
the large time behavior of smooth solutions with small perturbations
were obtained by Peng, Wang and Gu \cite{PWG}, Duan \cite{DRJ,DLZ},
Ueda, Wang and Kawashima \cite{UK,UWK}. In addition, the asymptotic
limits such as the non-relativistic limit
$(\varepsilon\rightarrow0)$, the quasi-neutral limit
$(\lambda\rightarrow0)$ and the combined non-relativistic and
quasi-neutral limits $(\varepsilon=\lambda\rightarrow0)$ have been
justified by Peng and Wang \cite{PW1,PW2,PW3}. The reader is also
referred to \cite{XX} for combined diffusive relaxation limits and
\cite{T1,T2} for WKB asymptotics; and references therein.

Up to now, the study for Euler-Maxwell equations in several
dimensions are still far from well known in the framework of
critical spaces. Recently, using the low- and high-frequency
decomposition arguments, we constructed uniform (global) classical
solutions (around constant equilibrium) to the Cauchy problem of
one-fluid Euler-Maxwell system in Chemin-Lerner spaces with critical
regularity. Furthermore, based on the Aubin-Lions compactness lemma,
it is justified that the (scaled) classical solutions converge
globally in time to the solutions of compressible Euler-Poisson
equations in the process of nonrelativistic limit and to that of
drift-diffusion equations under the relaxation limit or the combined
nonrelativistic and relaxation limits, see \cite{X2}.

In the present paper, we extend those results in \cite{X2} to the
two-fluid Euler-Maxwell equations (\ref{R-E1}). More precisely, we
consider the perturbation near the constant equilibrium state
$(1,0,1,0,0,\bar{B})(\bar{B}\in \Omega)$ which is a particular
solution of the system (\ref{R-E1})-(\ref{R-E2}), and achieve local
well-posedness for general data and global well-posedness for small
data. It should be pointed out that (\ref{R-E1}) is different from
the one-fluid case and this extension is not trivial. We are faced
with new difficulties arising from the more complicated nonlinear
coupling and cancelation between two carriers. For instance, the
expected dissipation rates for the densities of electrons and ions
are absent in whole space $\mathbb{R}^N$, and we only capture the
\textit{weaker} dissipation ones from contributions of $(\nabla
n_{+},\nabla n_{-})$ and $n_{+}-n_{-}$. Therefore, in order to close
the ``\textit{a priori}" estimates in critical spaces, new
techniques in comparison with \cite{X2} are adopted. Indeed, we
perform the \textit{homogeneous} blocks rather than the
\textit{inhomogeneous} blocks to localize the symmetric system, as
one captures the dissipation rate for velocities. Furthermore, the
elementary fact established in the recent work \cite{XK}, which
indicates the relations between homogeneous Chemin-Lerner spaces and
inhomogeneous Chemin-Lerner spaces, will been used. In addition,
different from that in \cite{X2}, we modify the nonlinear smooth
function arising from the symmetrization a little such that
$h(0)=h'(0)=0$, then we take full advantage of the continuity for
compositions in space-time Besov spaces (Chemin-Lerner spaces) which
is a natural generalization from Besov spaces to Chemin-Lerner
spaces, to estimate the cancelation of densities between two
carriers effectively. For above details, see Sect. \ref{sec:3},
Lemma \ref{lem4.1}-\ref{lem4.2} and Proposition
\ref{prop5.1}-\ref{prop5.2}.

To state main results more explicitly, we first introduce the
functional spaces
$$\widetilde{\mathcal{C}}_{T}(B^{s}_{p,r}(\Omega)):=\widetilde{L}^{\infty}_{T}(B^{s}_{p,r}(\Omega))\cap\mathcal{C}([0,T],B^{s}_{p,r}(\Omega))
$$ and $$\widetilde{\mathcal{C}}^1_{T}(B^{s}_{p,r}(\Omega)):=\{f\in\mathcal{C}^1([0,T],B^{s}_{p,r}(\Omega))|\partial_{t}f\in\widetilde{L}^{\infty}_{T}(B^{s}_{p,r}(\Omega))\},$$
where the index $T>0$ will be omitted when $T=+\infty$, the reader
is referred to Definition \ref{defn2.1} below for Chemin-Lerner
spaces.

Throughout this paper, let us denote by $s_{c}$ the critical number
$1+N/2$. First of all, we give the local existence and blow-up
criterion of classical solutions to (\ref{R-E1})-(\ref{R-E2}) away
from the vacuum.
\begin{thm}\label{thm1.1}
Let $\bar{B}\in\Omega$ be any given constant. Suppose that
$n_{\pm0}-1, u_{\pm0}, E_{0}$ and $B_{0}-\bar{B}\in
B^{s_{c}}_{2,1}(\Omega)$ satisfy $n_{\pm0}>0$ and the compatible
conditions (\ref{R-E3}). Then there exists a time $T_{0}>0$  such
that
\begin{itemize}
\item[(i)] Existence: the
system (\ref{R-E1})-(\ref{R-E2})  has a unique solution
$(n_{\pm},u_{\pm},E,B)\in\mathcal{C}^{1}([0,T_{0}]\times \Omega)$
with $n_{\pm}>0$ for all $t\in[0,T_{0}]$ and
$(n_{\pm}-1,u_{\pm},E,B-\bar{B})\in\widetilde{\mathcal{C}}_{T_{0}}(B^{s_{c}}_{2,1}(\Omega))\cap
\widetilde{\mathcal{C}}^1_{T_{0}}(B^{s_{c}-1}_{2,1}(\Omega))$;
\item[(ii)]Blow-up criterion: if the maximal time $T^{*}(>T_{0})$ of existence of such a solution
is finite, then
$$\limsup_{t\rightarrow T^{*}}\|n_{\pm}(t,\cdot)-1,u_{\pm}(t,\cdot),E(t,\cdot),B(t,\cdot)-\bar{B}\|_{B^{s_{c}}_{2,1}(\Omega)}=\infty$$
if and only if $$\int^{T^{*}}_{0}\|(\nabla n_{\pm},\nabla
u_{\pm},\nabla E,\nabla
B)(t,\cdot)\|_{L^{\infty}(\Omega)}dt=\infty.$$
\end{itemize}
\end{thm}

\begin{rem}
Recently, Xu and Kawashima \cite{XK} have established a general
theory on the well-posedness of generally symmetriable hyperbolic
systems in the framework of critical Chemin-Lerner spaces, which is
regarded as the generalization of the classical local existence
theory of Kato and Majda \cite{K,M}. As a matter of fact, the
results are also adapted to the periodic case. As in
Sect.~\ref{sec:3}, we see that (\ref{R-E1}) is transformed into a
symmetric hyperbolic system equivalently. Hence, the general theory
can be applied to the Euler-Maxwell equations. It is worth noting
that the blow-up criterion of classical solutions to the
Euler-Maxwell equations is obtained firstly in the present paper.
\end{rem}

In small amplitude regime, we establish the following global
well-posedness to (\ref{R-E1})-(\ref{R-E2}) in critical spaces.

\begin{thm}\label{thm1.2} Let $\bar{B}\in \Omega$ be any given constant.
Suppose that $(n_{\pm0}-1,u_{\pm0},E_{0},B_{0}-\bar{B})\in
B^{s_{c}}_{2,1}(\Omega)$ satisfy the compatible conditions
(\ref{R-E3}). There exists a positive constant $\delta_{0}$ such
that if
$$\|(n_{\pm0}-1,u_{\pm0},E_{0},B_{0}-\bar{B})\|_{B^{s_{c}}_{2,1}(\Omega)}\leq
\delta_{0},$$ then the system (\ref{R-E1})-(\ref{R-E2}) admits a
unique global solution $(n_{\pm},u_{\pm},E,B)$ satisfying
\begin{eqnarray*}
(n_{\pm},u_{\pm},E,B)\in \mathcal{C}^{1}([0,\infty)\times \Omega)
\end{eqnarray*}
and
$$(n_{\pm}-1,u_{\pm},E,B-\bar{B}) \in\widetilde{\mathcal{C}}(B^{s_{c}}_{2,1}(\Omega))\cap\widetilde{\mathcal{C}}^{1}(B^{s_{c}-1}_{2,1}(\Omega)).$$
Moreover, there are two positive constants $\mu_{0}$ and $C_{0}$
such that
\begin{itemize}
\item[(i)] when $\Omega=\mathbb{R}^N$, it yields the
following \begin{eqnarray}
&&\|(n_{\pm}-1,u_{\pm},E,B-\bar{B})\|_{\widetilde{L}^{\infty}(B^{s_{c}}_{2,1}(\Omega))}\nonumber\\&&
+\mu_{0}\Big\{\|(n_{+}-n_{-},u_{\pm})\|_{\widetilde{L}^{2}(B^{s_{c}}_{2,1}(\Omega))}
+\|(\nabla n_{\pm},E)\|_{\widetilde{L}^{2}(B^{s_{c}-1}_{2,1}(\Omega))}+\|\nabla B\|_{\widetilde{L}^{2}(B^{s_{c}-2}_{2,1}(\Omega))}\Big\}\nonumber\\
&\leq&
C_{0}\|(n_{\pm0}-1,u_{\pm0},E_{0},B_{0}-\bar{B})\|_{B^{s_{c}}_{2,1}(\Omega)};
\label{R-E4}
\end{eqnarray}

\item[(ii)]when $\Omega=\mathbb{T}^{N}$, we further set
$\bar{n}_{\pm0}=1$, it yields the following
\begin{eqnarray}
&&\|(n_{\pm}-1,u_{\pm},E,B-\bar{B})\|_{\widetilde{L}^{\infty}(B^{s_{c}}_{2,1}(\Omega))}\nonumber\\&&
+\mu_{0}\Big\{\|(n_{\pm}-1,u_{\pm})\|_{\widetilde{L}^{2}(B^{s_{c}}_{2,1}(\Omega))}
+\|E\|_{\widetilde{L}^{2}(B^{s_{c}-1}_{2,1}(\Omega))}+\|\nabla B\|_{\widetilde{L}^{2}(B^{s_{c}-2}_{2,1}(\Omega))}\Big\}\nonumber\\
&\leq&
C_{0}\|(n_{\pm0}-1,u_{\pm0},E_{0},B_{0}-\bar{B})\|_{B^{s_{c}}_{2,1}(\Omega)},
\label{R-E444}
\end{eqnarray}
\end{itemize}
where $\bar{f}$ denotes the mean value of $f(x)$ over
$\mathbb{T}^{N}$, that is,
$$\bar{f}=\frac{1}{|\mathbb{T}^{N}|}\int_{\mathbb{T}^{N}}f(x)dx.$$
\end{thm}
\begin{rem}
Following from approaches in the current paper, the well-posedness
results to the Cauchy problem and periodic problem pertaining to
data in the \textit{supercritical} Besov spaces
$B^s_{p,r}(\Omega)(s>s_{c}, p=2,\ 1\leq r\leq\infty)$ can be also
established. Furthermore, the fact that Sobolev spaces
$H^{s}(\Omega):=B^s_{2,2}(\Omega)$ allows the results to be also
true in the usual Sobolev spaces with $s>s_{c}$.
\end{rem}

\begin{rem}
In the whole space, the energy inequality (\ref{R-E4}) is not so
surprising in comparison with the one-fluid case as in \cite{X2},
however, it is indeed \textit{different}. Due to the nonlinear
coupling and cancelation between two carriers, the dissipation rates
of $(n_{+}, n_{-})$ does not appear in (\ref{R-E4}) any more, and
the dissipation rates from $n_{+}-n_{-}$ and $(\nabla n_{+},\nabla
n_{-})$ are available only. In this case, to overcome the technical
difficulties occurring in the \textit{a priori} estimates, some
useful facts in Chemin-Lerner spaces are developed. It is worth
noting that the dissipation rate of $(n_{+}, n_{-})$ itself in the
periodic case can be obtained, see the proof of Theorem
\ref{thm1.2}. Besides, from (\ref{R-E4})-(\ref{R-E444}), we see that
there is a ``1-regularity-loss" phenomenon for the dissipation rates
of electromagnetic field $(E,B)$.
\end{rem}

As a direct consequence of Theorem \ref{thm1.2}, we obtain the
large-time asymptotic behavior of global solutions near the
equilibrium state $(1,0,1,0,0,\bar{B})$ in some Besov spaces.

\begin{cor}\label{cor1.3}
Let $(n_{\pm},u_{\pm},E,B)$ be the global-in-time solution in
Theorem \ref{thm1.2}, it holds that ($\varepsilon>0$)
$$\|n_{+}(t,\cdot)-n_{-}(t,\cdot),u_{\pm}(t,\cdot)\|_{B^{s_{c}-\varepsilon}_{2,1}(\Omega)}\rightarrow
0,$$
$$\|E(t,\cdot)\|_{B^{s_{c}-1-\varepsilon}_{2,1}(\Omega)}\rightarrow
0, \ \
\|B(t,\cdot)-\bar{B}\|_{B^{s_{c}-2-\varepsilon}_{p,1}(\Omega)}\rightarrow
0, $$ moreover,
$$\|n_{\pm}(t,\cdot)-1\|_{B^{s_{c}-1-\varepsilon}_{p,1}(\mathbb{R}^{N})}\rightarrow
0 \ \ \Big(p=\frac{2N}{N-2}, \ N>2\Big), $$
$$\|n_{\pm}(t,\cdot)-1\|_{B^{s_{c}-\varepsilon}_{2,1}(\mathbb{T}^{N})}\rightarrow
0,$$ as the time variable $t\rightarrow +\infty$.
\end{cor}

\begin{rem}\label{rem1.3}
Recalling the Corollary 5.1 in \cite{FX}, we omit details of the
proof of Corollary \ref{cor1.3}, since they are similarly followed
by the Gagliardo-Nirenberg-Sobolev inequality (see, \textit{e.g.},
\cite{E}) and interpolation arguments. In addition,  from the
embedding $B^{s_{c}-\varepsilon}_{2,1}\hookrightarrow
B^{s_{c}-1-\varepsilon}_{p,1}(N=3)$, we know that the large-time
asymptotic behavior of densities $n_{\pm}$ in the whole space case
is \textit{weaker} than that in the periodic case.
\end{rem}

\begin{rem}\label{rem1.4}
Following from the similar manners, the corresponding results can be
obtained for non-isentropic two-fluid Euler-Maxwell equations, which
include the temperature transport equations of carriers rather than
the assumed pressure-density relations as in (\ref{R-E1}) only. Let
us mention that the dissipation rates of temperatures will behave as
that of velocities, that is, there is no regularity-loss phenomenon
for temperatures.
\end{rem}

The rest of this paper unfolds as follows. In Sect. \ref{sec:2}, we
briefly review some useful properties on Besov spaces. In Sect.
\ref{sec:3}, we establish the local existence and blow-up criterion
for the Euler-Maxwell equations (\ref{R-E1}). Sect. \ref{sec:4} is
devoted to the global existence of classical solutions in critical
spaces. In the last section (Sect. \ref{sec:5}), we remark a natural
generalization on the continuity of composition functions in
Chemin-Lerner spaces.

\section{Littlewood-Paley theory and functional spaces}\label{sec:2}
Throughout the paper, $f\lesssim g$ denotes $f\leq Cg$, where $C>0$
is a generic constant. $f\thickapprox g$ means $f\lesssim g$ and
$g\lesssim f$. Denote by $\mathcal{C}([0,T],X)$ (resp.,
$\mathcal{C}^{1}([0,T],X)$) the space of continuous (resp.,
continuously differentiable) functions on $[0,T]$ with values in a
Banach space $X$.  Also, $\|(f,g,h)\|_{X}$ means $
\|f\|_{X}+\|g\|_{X}+\|h\|_{X}$, where $f,g,h\in X$. $\langle
f,g\rangle$ denotes the inner product of two functions $f,g$ in
$L^2(\mathbb{R}^{N})$.

In this section, we briefly review the Littlewood-Paley
decomposition and some properties of Besov spaces. The reader is
also referred to, \textit{e.g.}, \cite{BCD,D} for more details.

Let us start with the Fourier transform. The Fourier transform
$\hat{f}$ of a $L^1$-function $f$ is given by
$$\mathcal{F}f=\int_{\mathbb{R}^{N}}f(x)e^{-2\pi x\cdot\xi}dx.$$ More
generally, the Fourier transform of any $f\in\mathcal{S}'$, the
space of tempered distributions, is given by
$$(\mathcal{F}f,g)=(f,\mathcal{F}g)$$ for any $g\in \mathcal{S}$, the Schwartz
class.

First, we fix some notation.
$$\mathcal{S}_{0}=\Big\{\phi\in\mathcal{S},\partial^{\alpha}\mathcal{F}f(0)=0, \forall \alpha \in \mathbb{N}^{N}\ \mbox{multi-index}\Big\}.$$
Its dual is given by
$$\mathcal{S}'_{0}=\mathcal{S}'/\mathcal{P},$$ where $\mathcal{P}$
is the space of polynomials.

We now introduce a dyadic partition of $\mathbb{R}^{N}$. We choose
$\phi_{0}\in \mathcal{S}$ such that $\phi_{0}$ is even,
$$\mathrm{supp}\phi_{0}:=A_{0}=\Big\{\xi\in\mathbb{R}^{N}:\frac{3}{4}\leq|\xi|\leq\frac{8}{3}\Big\},\  \mbox{and}\ \ \phi_{0}>0\ \ \mbox{on}\ \ A_{0}.$$
Set $A_{q}=2^{q}A_{0}$ for $q\in\mathbb{Z}$. Furthermore, we define
$$\phi_{q}(\xi)=\phi_{0}(2^{-q}\xi)$$ and define $\Phi_{q}\in
\mathcal{S}$ by
$$\mathcal{F}\Phi_{q}(\xi)=\frac{\phi_{q}(\xi)}{\sum_{q\in \mathbb{Z}}\phi_{q}(\xi)}.$$
It follows that both $\mathcal{F}\Phi_{q}(\xi)$ and $\Phi_{q}$ are
even and satisfy the following properties:
$$\mathcal{F}\Phi_{q}(\xi)=\mathcal{F}\Phi_{0}(2^{-q}\xi),\ \ \ \mathrm{supp}\ \mathcal{F}\Phi_{q}(\xi)\subset A_{q},\ \ \ \Phi_{q}(x)=2^{qN}\Phi_{0}(2^{q}x)$$
and
$$\sum_{q=-\infty}^{\infty}\mathcal{F}\Phi_{q}(\xi)=\cases{1,\ \ \ \mbox{if}\ \ \xi\in\mathbb{R}^{N}\setminus \{0\},
\cr 0, \ \ \ \mbox{if}\ \ \xi=0.}
$$
As a consequence, for any $f\in S'_{0},$ we have
$$\sum_{q=-\infty}^{\infty}\Phi_{q}\ast f=f.$$

To define the homogeneous Besov spaces, we set
$$\dot{\Delta}_{q}f=\Phi_{q}\ast f,\ \ \ \ q=0,\pm1,\pm2,...$$

\begin{defn}\label{defn2.1}
For $s\in \mathbb{R}$ and $1\leq p,r\leq\infty,$ the homogeneous
Besov spaces $\dot{B}^{s}_{p,r}$ is defined by
$$\dot{B}^{s}_{p,r}=\{f\in S'_{0}:\|f\|_{\dot{B}^{s}_{p,r}}<\infty\},$$
where
$$\|f\|_{\dot{B}^{s}_{p,r}}
=\cases{\Big(\sum_{q\in\mathbb{Z}}(2^{qs}\|\dot{\Delta}_{q}f\|_{L^p})^{r}\Big)^{1/r},\
\ r<\infty, \cr \sup_{q\in\mathbb{Z}}
2^{qs}\|\dot{\Delta}_{q}f\|_{L^p},\ \ r=\infty.} $$\end{defn}

To define the inhomogeneous Besov spaces, we set $\Psi\in
\mathcal{C}_{0}^{\infty}(\mathbb{R}^{N})$ be even and satisfy
$$\mathcal{F}\Psi(\xi)=1-\sum_{q=0}^{\infty}\mathcal{F}\Phi_{q}(\xi).$$
It is clear that for any $f\in S'$, yields
$$\Psi*f+\sum_{q=0}^{\infty}\Phi_{q}\ast f=f.$$
We further set
$$\Delta_{q}f=\cases{0,\ \ \ \ \ \ \ \, \ j\leq-2,\cr
\Psi*f,\ \ \ j=-1,\cr \Phi_{q}\ast f, \ \ j=0,1,2,...}$$

\begin{defn}\label{defn2.2}
For $s\in \mathbb{R}$ and $1\leq p,r\leq\infty,$ the inhomogeneous
Besov spaces $B^{s}_{p,r}$ is defined by
$$B^{s}_{p,r}=\{f\in S':\|f\|_{B^{s}_{p,r}}<\infty\},$$
where
$$\|f\|_{B^{s}_{p,r}}
=\cases{\Big(\sum_{q=-1}^{\infty}(2^{qs}\|\dot{\Delta}_{q}f\|_{L^p})^{r}\Big)^{1/r},\
\ r<\infty, \cr \sup_{q\geq-1} 2^{qs}\|\dot{\Delta}_{q}f\|_{L^p},\ \
r=\infty.}$$
\end{defn}

Let us point out that the definitions of $\dot{B}^{s}_{p,r}$ and
$B^{s}_{p,r}$ does not depend on the choice of the Littlewood-Paley
decomposition.  Now, we state some basic properties, which will be
used in subsequent analysis.
\begin{lem}(Bernstein inequality)\label{lem2.1}
Let $k\in\mathbb{N}$ and $0<R_{1}<R_{2}$. There exists a constant
$C$, depending only on $R_{1},R_{2}$ and $N$, such that for all
$1\leq a\leq b\leq\infty$ and $f\in L^{a}$,
$$
\mathrm{Supp}\mathcal{F}f\subset \{\xi\in \mathbb{R}^{N}: |\xi|\leq
R_{1}\lambda\}\Rightarrow\sup_{|\alpha|=k}\|\partial^{\alpha}f\|_{L^{b}}
\leq C^{k+1}\lambda^{k+d(\frac{1}{a}-\frac{1}{b})}\|f\|_{L^{a}};
$$
$$
\mathrm{Supp}\mathcal{F}f\subset \{\xi\in \mathbb{R}^{N}:
R_{1}\lambda\leq|\xi|\leq R_{2}\lambda\} \Rightarrow
C^{-k-1}\lambda^{k}\|f\|_{L^{a}}\leq
\sup_{|\alpha|=k}\|\partial^{\alpha}f\|_{L^{a}}\leq
C^{k+1}\lambda^{k}\|f\|_{L^{a}}.
$$
\end{lem}

As a direct corollary of the above inequality, we have
\begin{rem}\label{rem2.1} For all
multi-index $\alpha$, it holds that
$$\frac{1}{C}\|f\|_{\dot{B}^{s + |\alpha|}_{p,
r}}\leq\|\partial^\alpha f\|_{\dot{B}^s_{p, r}}\leq
C\|f\|_{\dot{B}^{s + |\alpha|}_{p, r}};$$
$$
\|\partial^\alpha f\|_{B^s_{p, r}}\leq C\|f\|_{B^{s + |\alpha|}_{p,
r}}.
$$
\end{rem}

The second one is the embedding properties in Besov spaces.
\begin{lem}\label{lem2.2} Let $s\in \mathbb{R}$ and $1\leq
p,r\leq\infty,$ then
\begin{itemize}
\item[(1)]If $s>0$, then $B^{s}_{p,r}=L^{p}\cap \dot{B}^{s}_{p,r};$
\item[(2)]If $\tilde{s}\leq s$, then $B^{s}_{p,r}\hookrightarrow
B^{\tilde{s}}_{p,\tilde{r}}$;
\item[(3)]If $1\leq r\leq\tilde{r}\leq\infty$, then $\dot{B}^{s}_{p,r}\hookrightarrow
\dot{B}^{s}_{p,\tilde{r}}$ and $B^{s}_{p,r}\hookrightarrow
B^{s}_{p,\tilde{r}};$
\item[(4)]If $1\leq p\leq\tilde{p}\leq\infty$, then $\dot{B}^{s}_{p,r}\hookrightarrow \dot{B}^{s-N(\frac{1}{p}-\frac{1}{\tilde{p}})}_{\tilde{p},r}
$ and $B^{s}_{p,r}\hookrightarrow
B^{s-N(\frac{1}{p}-\frac{1}{\tilde{p}})}_{\tilde{p},r}$;
\item[(5)]$\dot{B}^{N/p}_{p,1}\hookrightarrow\mathcal{C}_{0},\ \ B^{N/p}_{p,1}\hookrightarrow\mathcal{C}_{0}(1\leq p<\infty);$
\end{itemize}
where $\mathcal{C}_{0}$ is the space of continuous bounded functions
which decay at infinity.
\end{lem}

On the other hand, we also present the definition of Chemin-Lerner
space-time spaces first introduced by J.-Y. Chemin and N. Lerner
\cite{C2}, which are the refinement of the spaces
$L^{\theta}_{T}(\dot{B}^{s}_{p,r})$ or
$L^{\theta}_{T}(B^{s}_{p,r})$.

\begin{defn}\label{defn2.3}
For $T>0, s\in\mathbb{R}, 1\leq r,\theta\leq\infty$, the homogeneous
mixed time-space Besov spaces
$\widetilde{L}^{\theta}_{T}(\dot{B}^{s}_{p,r})$ is defined by
$$\widetilde{L}^{\theta}_{T}(\dot{B}^{s}_{p,r}):
=\{f\in
L^{\theta}(0,T;\mathcal{S}'_{0}):\|f\|_{\widetilde{L}^{\theta}_{T}(\dot{B}^{s}_{p,r})}<+\infty\},$$
where
$$\|f\|_{\widetilde{L}^{\theta}_{T}(\dot{B}^{s}_{p,r})}:=\Big(\sum_{q\in\mathbb{Z}}(2^{qs}\|\dot{\Delta}_{q}f\|_{L^{\theta}_{T}(L^{p})})^{r}\Big)^{\frac{1}{r}}$$
with the usual convention if $r=\infty$.
\end{defn}

\begin{defn}\label{defn2.4}
For $T>0, s\in\mathbb{R}, 1\leq r,\theta\leq\infty$, the
inhomogeneous mixed time-space Besov spaces
$\widetilde{L}^{\theta}_{T}(B^{s}_{p,r})$ is defined by
$$\widetilde{L}^{\theta}_{T}(B^{s}_{p,r}):
=\{f\in
L^{\theta}(0,T;\mathcal{S}'):\|f\|_{\widetilde{L}^{\theta}_{T}(B^{s}_{p,r})}<+\infty\},$$
where
$$\|f\|_{\widetilde{L}^{\theta}_{T}(B^{s}_{p,r})}:=\Big(\sum_{q\geq-1}(2^{qs}\|\Delta_{q}f\|_{L^{\theta}_{T}(L^{p})})^{r}\Big)^{\frac{1}{r}}$$
with the usual convention if $r=\infty$.
\end{defn}

Next we state some basic properties on the inhomogeneous
Chemin-Lerner spaces only, since the similar ones are true in the
homogeneous Chemin-Lerner spaces.

The first one is that $\widetilde{L}^{\theta}_{T}(B^{s}_{p,r})$ may
be linked with the classical spaces $L^{\theta}_{T}(B^{s}_{p,r})$
via the Minkowski's inequality:
\begin{rem}\label{rem2.2}
It holds that
$$\|f\|_{\widetilde{L}^{\theta}_{T}(B^{s}_{p,r})}\leq\|f\|_{L^{\theta}_{T}(B^{s}_{p,r})}\,\,\,
\mbox{if}\,\, r\geq\theta;\ \ \ \
\|f\|_{\widetilde{L}^{\theta}_{T}(B^{s}_{p,r})}\geq\|f\|_{L^{\theta}_{T}(B^{s}_{p,r})}\,\,\,
\mbox{if}\,\, r\leq\theta.
$$\end{rem}
Let us also recall the property of continuity for product in
Chemin-Lerner spaces $\widetilde{L}^{\theta}_{T}(B^{s}_{p,r})$.
\begin{prop}\label{prop2.1}
The following inequality holds:
$$
\|fg\|_{\widetilde{L}^{\theta}_{T}(B^{s}_{p,r})}\leq
C(\|f\|_{L^{\theta_{1}}_{T}(L^{\infty})}\|g\|_{\widetilde{L}^{\theta_{2}}_{T}(B^{s}_{p,r})}
+\|g\|_{L^{\theta_{3}}_{T}(L^{\infty})}\|f\|_{\widetilde{L}^{\theta_{4}}_{T}(B^{s}_{p,r})})
$$
whenever $s>0, 1\leq p\leq\infty,
1\leq\theta,\theta_{1},\theta_{2},\theta_{3},\theta_{4}\leq\infty$
and
$$\frac{1}{\theta}=\frac{1}{\theta_{1}}+\frac{1}{\theta_{2}}=\frac{1}{\theta_{3}}+\frac{1}{\theta_{4}}.$$
As a direct corollary, one has
$$\|fg\|_{\widetilde{L}^{\theta}_{T}(B^{s}_{p,r})}
\leq
C\|f\|_{\widetilde{L}^{\theta_{1}}_{T}(B^{s}_{p,r})}\|g\|_{\widetilde{L}^{\theta_{2}}_{T}(B^{s}_{p,r})}$$
whenever $s\geq d/p,
\frac{1}{\theta}=\frac{1}{\theta_{1}}+\frac{1}{\theta_{2}}.$
\end{prop}

In the next symmetrization, we meet with some composition functions.
The following continuity result for compositions is used to estimate
them.
\begin{prop}\label{prop2.2}(\cite{A})
Let $s>0$, $1\leq p, r, \theta\leq \infty$, $F'\in
W^{[s]+1,\infty}_{loc}(I;\mathbb{R})$ with $F(0)=0$, $T\in
(0,\infty]$ and $v\in \widetilde{L}^{\theta}_{T}(B^{s}_{p,r})\cap
L^{\infty}_{T}(L^{\infty}).$ Then
$$\|F(f)\|_{\widetilde{L}^{\theta}_{T}(B^{s}_{p,r})}\leq
C(1+\|f\|_{L^{\infty}_{T}(L^{\infty})})^{[s]+1}\|F'\|_{W^{[s]+1,\infty}}\|f\|_{\widetilde{L}^{\theta}_{T}(B^{s}_{p,r})}.$$
\end{prop}

In addition, we present some estimates of commutators in homogeneous
and inhomogeneous Chemin-Lerner spaces to bound commutators.
\begin{prop}\label{prop2.3}
Let  $1<p<\infty$ and $1\leq \rho\leq\infty$. Then there exists a
generic constant $C>0$ depending only on $s, N$ such that
$$\cases{\|[f,\Delta_{q}]\mathcal{A}g\|_{L^{\theta}_{T}(L^{p})}\leq
Cc_{q}2^{-qs}\|\nabla
f\|_{\widetilde{L}^{\theta_{1}}_{T}(B^{s-1}_{p,1})}\|g\|_{\widetilde{L}^{\theta_{2}}_{T}(B^{s}_{p,1})},\
\ s=1+\frac{N}{p},\cr
\|[f,\Delta_{q}]g\|_{L^{\theta}_{T}(L^{p})}\leq
Cc_{q}2^{-q(s+1)}\|f\|_{\widetilde{L}^{\theta_{1}}_{T}(\dot{B}^{\frac{N}{p}+1}_{p,1})}\|g\|_{\widetilde{L}^{\theta_{2}}_{T}(\dot{B}^{s}_{p,1})},\
\ s\in(-\frac{N}{p}-1, \frac{N}{p}],}
$$
where the commutator $[\cdot,\cdot]$ is defined by $[f,g]=fg-gf$,
and the operator $\mathcal{A}:=\mathrm{div}$ or $\mathrm{\nabla}$.
$\{c_{q}\}$ denotes a sequence such that $\|(c_{q})\|_{ {l^{1}}}\leq
1$, $\frac{1}{\theta}=\frac{1}{\theta_{1}}+\frac{1}{\theta_{2}}$.
\end{prop}

Finally, let us point out that all the properties described in the
this section remain true in the periodic setting, see \cite{D}.

\section{Local existence and blow-up criterion}\setcounter{equation}{0} \label{sec:3}

It is convenient to obtain the main results, we first reformulate
the compressible Euler-Maxwell system (\ref{R-E1}). Set
\begin{equation}
\left\{
\begin{array}{l}
\varrho_{\pm}(t,x)=\frac {2}{\gamma-1}\{[n_{\pm}(\frac
{t}{\sqrt{\gamma}},x)]^{\frac {\gamma-1}{2}}-1\},\ \ \
\upsilon_{\pm}(t,x)=\frac {1}{\sqrt{\gamma}}u_{\pm}(\frac {t}{\sqrt{\gamma}},x),\\
\tilde{E}(t,x)=\frac {1}{\sqrt{\gamma}}E(\frac
{t}{\sqrt{\gamma}},x),\ \ \  \tilde{B}(t,x)=\frac
{1}{\sqrt{\gamma}}B(\frac {t}{\sqrt{\gamma}},x)-\bar{B}.
 \end{array}
\right.\label{R-E5}
\end{equation}
Then the system (\ref{R-E1}) can be reformulated, for classical
solution $W=(\varrho_{\pm},\upsilon_{\pm},\tilde{E},\tilde{B})$, as
\begin{equation}
\left\{
\begin{array}{l}
\partial_{t} \varrho_{\pm}+\upsilon_{\pm}\cdot\nabla\varrho_{\pm}+(\frac {\gamma-1}{2}\varrho_{\pm}+1)\nabla\cdot\upsilon_{\pm}=0,\\
\partial_{t} \upsilon_{\pm}+(\frac {\gamma-1}{2}\varrho_{\pm}+1)\nabla\varrho_{\pm}+\upsilon_{\pm}\cdot\nabla\upsilon_{\pm}=
   \mp(\frac {1}{\sqrt{\gamma}}\tilde{E}+\upsilon_{\pm}\times (\tilde{B}+\bar{B}))-\frac {1}{\sqrt{\gamma}}\upsilon_{\pm},\\
\partial_{t} \tilde{E}-\frac {1}{\sqrt{\gamma}}\nabla\times\tilde{B}=\frac {1}{\sqrt{\gamma}}\upsilon_{+}+
    \frac {1}{\sqrt{\gamma}}[\Phi(\varrho_{+})+\varrho_{+}]\upsilon_{+}-\frac {1}{\sqrt{\gamma}}\upsilon_{-}-
    \frac {1}{\sqrt{\gamma}}[\Phi(\varrho_{-})+\varrho_{-}]\upsilon_{-},\\
\partial_{t} \tilde{B}+\frac {1}{\sqrt{\gamma}}\nabla\times\tilde{E}=0,\\
\nabla\cdot\tilde{E}=-\frac
{1}{\sqrt{\gamma}}[\Phi(\varrho_{+})+\varrho_{+}]+\frac
{1}{\sqrt{\gamma}}[\Phi(\varrho_{-})+\varrho_{-}],\ \ \
\nabla\cdot\tilde{B}=0
 \end{array}
\right.\label{R-E6}
\end{equation}
with the initial data
\begin{equation}
W|_{t=0}=W_{0}:=(\varrho_{\pm0},\upsilon_{\pm0},\tilde{E}_{0},\tilde{B}_{0})\label{R-E7}
\end{equation}
satisfying the corresponding compatible conditions
\begin{equation}
\left\{
\begin{array}{l}
\nabla\cdot\tilde{E}_{0}=-\frac {1}{\sqrt{\gamma}}[\Phi(\varrho_{+0})+\varrho_{+0}]+\frac {1}{\sqrt{\gamma}}[\Phi(\varrho_{-0})+\varrho_{-0}],\\
\nabla\cdot\tilde{B}_{0}=0.
 \end{array}
\right.\label{R-E8}
\end{equation}
Here the nonlinear function $\Phi(\cdot)$ in (\ref{R-E6}) is defined
by
\[
\Phi(\rho)=(\frac {\gamma-1}{2}\rho+1)^{\frac {2}{\gamma-1}}-\rho-1.
\]
Notice that $\Phi(\rho)$ is a smooth function on the domain
$\{\rho|\frac {\gamma-1}{2}\rho+1>0\}$ satisfying
$\Phi(0)=\Phi'(0)=0$, which is a little different from that in
\cite{X2}.

For this reformulation, we have
\begin{rem}\label{rem3.1} The variable change is from the open set $\{(n_{+},u_{+},n_{-},u_{-},E,B)\in
(0,+\infty)\times \mathbb{R}^{N}\times (0,+\infty)\times
\mathbb{R}^{N} \times\mathbb{R}^{N}\times \mathbb{R}^{N}\}$ to the
open set $\{W\in\mathbb{R}\times \mathbb{R}^{N}\times
\mathbb{R}\times \mathbb{R}^{N}\times \mathbb{R}^{N}\times
\mathbb{R}^{N}|\frac{\gamma-1}{2}\varrho_{\pm}+1>0\}$. It is easy to
show that for classical solutions $(n_{\pm},u_{\pm},E,B)$ away from
vacuum, (\ref{R-E1})-(\ref{R-E2}) is equivalent to
(\ref{R-E6})-(\ref{R-E7}) with
$\frac{\gamma-1}{2}\varrho_{\pm}+1>0$.
\end{rem}
The simpler case of $\gamma=1$ can be treated in the similar way by
using the reformulation in terms of the enthalpy variable, see,
\textit{e.g.}, \cite{X2}. Without loss of generality, we focus on
the system (\ref{R-E6})-(\ref{R-E7}).

Next, let us write (\ref{R-E6}) as a symmetric hyperbolic system.
Set
$$W_{I}=(\varrho_{+},\upsilon_{+},\varrho_{-},\upsilon_{-})^{\top},\ W_{II}=(\tilde{E},\tilde{B})^{\top},\ W=(W_{I},W_{II})^{\top}.$$
Then (\ref{R-E6}) is reduced to
\begin{eqnarray}
\partial_{t}W+{\sum_{j=1}^{N}}A_{j}(W_{I})\partial_{x_{j}}W=L(W), \label{R-E9}
\end{eqnarray}
where
\begin{eqnarray*}
A_{j}(W_{I})=\left(%
\begin{array}{cc}
  A^{I}_{j}(W_{I}) & 0 \\
  0 & A^{II}_{j} \\
\end{array}%
\right),
\end{eqnarray*}
\begin{eqnarray*}
L(W)=\left(%
\begin{array}{c}
  0 \\
  -(\frac {1}{\sqrt{\gamma}}\tilde{E}+\upsilon_{+}\times (\tilde{B}+\bar{B}))-\frac {1}{\sqrt{\gamma}}\upsilon_{+}\\
  0\\
  (\frac {1}{\sqrt{\gamma}}\tilde{E}+\upsilon_{-}\times (\tilde{B}+\bar{B}))-\frac {1}{\sqrt{\gamma}}\upsilon_{-}\\
  \frac {1}{\sqrt{\gamma}}\upsilon_{+}+\frac {1}{\sqrt{\gamma}}[\Phi(\varrho_{+})+\varrho_{+}]\upsilon_{+}-\frac {1}{\sqrt{\gamma}}\upsilon_{-}
  -\frac {1}{\sqrt{\gamma}}[\Phi(\varrho_{-})+\varrho_{-}]\upsilon_{-}\\
  0\\ \end{array}%
  \right)
\end{eqnarray*}
with
\begin{eqnarray*}
 A^{I}_{j}(W_{I})=\left(%
 \begin{array}{cccc}
    \upsilon^{j}_{+} & (\frac {\gamma-1}{2}\varrho_{+}+1)e_{j}^{\top} & 0 & 0 \\
   (\frac {\gamma-1}{2}\varrho_{+}+1)e_{j} & \upsilon^{j}_{+}I_{N} & 0 & 0 \\
   0 & 0 & \upsilon^{j}_{-} & (\frac {\gamma-1}{2}\varrho_{-}+1)e_{j}^{\top} \\
   0 & 0 & (\frac {\gamma-1}{2}\varrho_{-}+1)e_{j} & \upsilon^{j}_{-}I_{N}  \\
 \end{array}%
 \right), \end{eqnarray*}
\begin{eqnarray*}
A^{II}_{j}=\left(%
\begin{array}{cc}
 0 & P_{j} \\
  P_{j}^{\top} & 0 \\
\end{array}%
\right)
\end{eqnarray*}
and
\begin{eqnarray*}
P_{1}=\left(%
\begin{array}{ccc}
  0 & 0 & 0  \\
  0 & 0 & \frac {1}{\sqrt{\gamma}} \\
  0 & -\frac {1}{\sqrt{\gamma}} & 0  \\
\end{array}%
\right),\
P_{2}=\left(%
\begin{array}{ccc}
  0 & 0 & -\frac {1}{\sqrt{\gamma}} \\
  0 & 0 & 0 \\
  \frac {1}{\sqrt{\gamma}} & 0 & 0  \\
\end{array}%
\right),\
P_{3}=\left(%
\begin{array}{ccc}
  0 & \frac {1}{\sqrt{\gamma}} & 0 \\
  -\frac {1}{\sqrt{\gamma}} & 0 & 0 \\
  0 & 0 & 0  \\
\end{array}%
\right).
\end{eqnarray*}

Here $I_{N}$ denotes the unit matrix of order $N$ and $e_{j}$ is the
$N$-dimensional vector where the $j$th component is one, others are
zero. From the explicit structure of the block matrix $A_{j}(W_{I})$
above, we see that (\ref{R-E6}) is a symmetric hyperbolic system on
$G=\{W|\frac{\gamma-1}{2}\varrho_{\pm}+1>0\}$ in the sense of
Friedrichs. Based on the recent work \cite{XK} for generally
symmetric hyperbolic systems, we get the local existence and
uniqueness of classical solutions
$W=(\varrho_{\pm},\upsilon_{\pm},\tilde{E},\tilde{B})$, which reads
as follows.

\begin{prop}\label{prop3.1} Assume that
$W_{0}\in{B^{s_{c}}_{2,1}}$ satisfying
$\frac{\gamma-1}{2}\varrho_{\pm0}+1>0$ and (\ref{R-E8}), then there
exists a time $T_{0}>0$ (depending only on the initial data) such
that
\begin{itemize}
\item[(i)] Existence: the
system (\ref{R-E6})-(\ref{R-E7}) has a unique solution
$W\in\mathcal{C}^{1}([0,T_{0}]\times \mathbb{R}^{N})$ with
$\frac{\gamma-1}{2}\varrho_{\pm}+1>0$ for all $t\in[0,T_{0}]$ and
$W\in\widetilde{\mathcal{C}}_{T_{0}}(B^{s_{c}}_{2,1})\cap
\widetilde{\mathcal{C}}^1_{T_{0}}(B^{s_{c}-1}_{2,1})$;
\item[(ii)]Blow-up criterion: if the maximal time $T^{*}(>T_{0})$ of existence of such a solution
is finite, then
$$\limsup_{t\rightarrow T^{*}}\|W(t)\|_{B^{s_{c}}_{2,1}}=\infty$$
if and only if $$\int^{T^{*}}_{0}\|\nabla
W(t)\|_{L^{\infty}}dt=\infty.$$
\end{itemize}
\end{prop}

\begin{proof}
From \cite{XK}, it suffices to establish the blow-up criterion. We
consider the symmetric system (\ref{R-E9}) with $L(W)\equiv0$ for
simplicity, since it is only responsible for the global
well-posedness and large time behavior of solutions.

Applying the homogeneous operator $\dot{\Delta}_{q}$ to
(\ref{R-E9}), we infer that $\dot{\Delta}_{q}W$ satisfies
\begin{equation}
\partial_{t}\dot{\Delta}_{q}W+\sum_{j=1}^{N}A_{j}(W_{I})\dot{\Delta}_{q}\partial_{x_{j}}W
=-\sum_{j=1}^{N}[\dot{\Delta}_{q},A_{j}(W_{I})]W_{x_{j}},
\label{R-E10}
\end{equation}
where the commutator $[\cdot,\cdot]$ is defined by $[f,g]:=fg-gf$.

Perform the inter product with $\dot{\Delta}_{q}W$ on both sides of
the equation (\ref{R-E10}) to get
\begin{eqnarray}
&&\langle\dot{\Delta}_{q}W,\dot{\Delta}_{q}W\rangle_{t}
+\sum_{j=1}^{N}\langle A_{j}(W_{I})\dot{\Delta}_{q}W,
\dot{\Delta}_{q}W\rangle_{x_{j}}
\nonumber\\&=&-2\sum_{j=1}^{N}\langle
[\dot{\Delta}_{q},A_{j}(W_{I})]W_{x_{j}},\dot{\Delta}_{q}W\rangle
+\langle A_{j}(W_{I})_{x_{j}}\dot{\Delta}_{q}W,
\dot{\Delta}_{q}W\rangle. \label{R-E11}
\end{eqnarray}
By integrating (\ref{R-E11}) with respect to $x$ over
$\mathbb{R}^{N}$, we get
\begin{eqnarray}
\frac{d}{dt}\|\dot{\Delta}_{q}W\|^2_{L^2}
\lesssim\|[\dot{\Delta}_{q},A_{j}(W_{I})]W_{x_{j}}\|_{L^2}\|\Delta_{q}W\|_{L^2}+
\|A_{j}(W_{I})_{x_{j}}\|_{L^\infty}\|\dot{\Delta}_{q}W\|^2_{L^2}.
\label{R-E12}
\end{eqnarray}
Let $\epsilon>0$ be a small number. Dividing (\ref{R-E12}) by
$(\|\dot{\Delta}_{q}W\|^2_{L^2}+\epsilon)^{1/2}$ gives
\begin{eqnarray}
&&\frac{d}{dt}\Big(\|\dot{\Delta}_{q}W\|^2_{L^2}+\epsilon\Big)^{1/2}\nonumber\\&\lesssim&
\|[\dot{\Delta}_{q},A_{j}(W_{I})]W_{x_{j}}\|_{L^2}+
\|A_{j}(W_{I})_{x_{j}}\|_{L^\infty}\|\dot{\Delta}_{q}W\|_{L^2}\nonumber\\&\lesssim&
c_{q}(t)2^{-qs_{c}}(\|A_{j}(W_{I})_{x_{j}}\|_{L^\infty}\|W\|_{\dot{B}^{s_{c}}_{2,1}}+\|A_{j}(W_{I})_{x_{j}}\|_{\dot{B}^{s_{c}}_{2,1}}\|\nabla
W\|_{L^\infty})\nonumber\\&&+\|A_{j}(W_{I})_{x_{j}}\|_{L^\infty}\|\dot{\Delta}_{q}W\|_{L^2},\label{R-E13}
\end{eqnarray}
where we used the stationary cases of estimates of commutator in
\cite{BCD}(Lemma 2.100, P.112) and the sequence $\{c_{q}(t)\}$
satisfying $\|c_{q}(t)\|_{\ell^1}\leq1$, for all $t\in[0,T_{0}]$.

Taking a time integration and passing to the limit
$\epsilon\rightarrow0$, we arrive at
\begin{eqnarray}
\|W(t)\|_{\dot{B}^{s_{c}}_{2,1}}&\lesssim&\|W_{0}\|_{\dot{B}^{s_{c}}_{2,1}}
+\int^{T_{0}}_{0}\|\nabla
W(\tau)\|_{L^\infty}\|W(\tau)\|_{\dot{B}^{s_{c}}_{2,1}}d\tau,\label{R-E14}
\end{eqnarray}
since we note that the fact $W(t,x)\in \mathcal{O}_{1}$(a bounded
open convex set in $\mathbb{R}^{4N+2}$) for any $(t,x)\in
[0,T_{0}]\times \mathbb{R}^{N}$, see \cite{XK}.

On the other hand, we take the $L^2$-inner product on (\ref{R-E9})
with $W$. It is not difficult to obtain
\begin{eqnarray}
\|W(t)\|_{L^2}&\lesssim&\|W_{0}\|_{L^2} +\int^{T_{0}}_{0}\|\nabla
W(\tau)\|_{L^\infty}\|W(\tau)\|_{L^2}d\tau.\label{R-E15}
\end{eqnarray}
Adding (\ref{R-E14}) to (\ref{R-E15}),  from (1) in Lemma
\ref{lem2.2}, we have
\begin{eqnarray}
\|W(t)\|_{B^{s_{c}}_{2,1}}&\lesssim&\|W_{0}\|_{B^{s_{c}}_{2,1}}
+\int^{T_{0}}_{0}\|\nabla
W(\tau)\|_{L^\infty}\|W(\tau)\|_{B^{s_{c}}_{2,1}}d\tau.\label{R-E16}
\end{eqnarray}
Gronwall's inequality implies
\begin{eqnarray}
\sup_{t\in[0,T_{0}]}\|W(t)\|_{B^{s_{c}}_{2,1}}\lesssim\|W_{0}\|_{B^{s_{c}}_{2,1}}\exp\Big(\int_{0}^{T_{0}}\|\nabla
W(\tau)\|_{L^\infty}d\tau\Big).\label{R-E17}
\end{eqnarray}
Besides, we have the following obvious inequalities
\begin{eqnarray}
\int^{T_{0}}_{0}\|\nabla
W(\tau)\|_{L^\infty}d\tau\lesssim\int^{T_{0}}_{0}\|W(\tau)\|_{B^{s_{c}}_{2,1}}d\tau
\lesssim
T_{0}\sup_{t\in[0,T_{0}]}\|W(t)\|_{B^{s_{c}}_{2,1}}.\label{R-E18}
\end{eqnarray}
Hence the blow-up criterion follows (\ref{R-E17}) and (\ref{R-E18})
immediately. This completes Proposition \ref{prop3.1}.
\end{proof}

\section{Global well-posedness}\setcounter{equation}{0} \label{sec:4}
In this section, we focus on the global existence of classical
solutions to (\ref{R-E6})-(\ref{R-E7}). For that purpose, we first
derive a crucial \textit{a priori} estimate in the whole space,
which is comprised in the following proposition.
\begin{prop}\label{prop4.1}
There exist some positive constants $\delta_{1}, \mu_{1}$ and
$C_{1}$ such that for any $T>0$, if
\begin{eqnarray}
\|(\varrho_{\pm},\upsilon_{\pm},\tilde{E},\tilde{B})\|_{\widetilde{L}^\infty_{T}(B^{s_{c}}_{2,1})}\leq
\delta_{1},\label{R-E19}
\end{eqnarray}
then
\begin{eqnarray}&&\|(\varrho_{\pm},\upsilon_{\pm},\tilde{E},\tilde{B})\|_{\widetilde{L}^\infty_{T}(B^{s_{c}}_{2,1})}
\nonumber\\&&+\mu_{1}\Big\{\|(\varrho_{+}-\varrho_{-},\upsilon_{\pm})\|_{\widetilde{L}^2_{T}(B^{s_{c}}_{2,1})}
+\|(\nabla\varrho_{\pm},\tilde{E})\|_{\widetilde{L}^2_{T}(B^{s_{c}-1}_{2,1})}
+\|\nabla\tilde{B}\|_{\widetilde{L}^2_{T}(B^{s_{c}-2}_{2,1})}\Big\}
\nonumber\\&\leq&
C_{1}\|(\varrho_{\pm0},\upsilon_{\pm0},\tilde{E}_{0},\tilde{B}_{0})\|_{B^{s_{c}}_{2,1}}.\label{R-E20}
\end{eqnarray}
\end{prop}
Actually, the proof of Proposition \ref{prop4.1} is to capture the
dissipation rates from contributions of
$(\varrho_{\pm},\upsilon_{\pm},\tilde{E},\tilde{B})$ in turn by
using the low- and high-frequency decomposition methods. For
clarity, we divide it into several lemmas.

\begin{lem}\label{lem4.1}
If $W\in\widetilde{\mathcal{C}}_{T}(B^{s_{c}}_{2,1})\cap
\widetilde{\mathcal{C}}^1_{T}(B^{s_{c}-1}_{2,1})$ is a solution of
(\ref{R-E6})-(\ref{R-E7}) for any $T>0$, then the following estimate
holds:
\begin{eqnarray}
&&\|W\|_{\widetilde{L}^\infty_{T}(B^{s_{c}}_{2,1})}
+\mu_{2}\|(\upsilon_{+},\upsilon_{-})\|_{\widetilde{L}^2_{T}(B^{s_{c}}_{2,1})}\nonumber
\\&\lesssim&\|W_{0}\|_{B^{s_{c}}_{2,1}}+\sqrt{\|W\|_{\widetilde{L}_{T}^{\infty}(B^{s_{c}}_{2,1})}}
\Big(\|(\nabla\varrho_{+},\nabla\varrho_{-},\widetilde{E})\|_{\widetilde{L}_{T}^2(B^{s_{c}-1}_{2,1})}+\|(\upsilon_{+},\upsilon_{-})\|_{\widetilde{L}_{T}^2(B^{s_{c}}_{2,1})}\Big),
\label{R-E22}
\end{eqnarray}
where $\mu_{2}$ is a positive constant.
\end{lem}
\begin{proof}
The proof is divided into three steps.\\

\noindent\textit{\underline{Step 1. The
$\widetilde{L}^2_{T}(\dot{B}^{s_{c}}_{2,1})$ estimates of
$(\upsilon_{+},\upsilon_{-})$}}.

Indeed, applying the homogeneous localization operator
${\dot{\Delta}}_{q}(q\in\mathbb{Z})$ to (\ref{R-E6}), we infer that
\begin{equation}
\left\{
\begin{array}{l}
\partial_{t} {\dot{\Delta}}_{q}\varrho_{+}+{\dot{\Delta}}_{q}\mathrm{div}\upsilon_{+}\\ \hspace{5mm}=-(\upsilon_{+}\cdot\nabla){\dot{\Delta}}_{q}\varrho_{+}+[\upsilon_{+},{\dot{\Delta}}_{q}]\cdot\nabla\varrho_{+}
-\frac
{\gamma-1}{2}([{\dot{\Delta}}_{q},\varrho_{+}]\mathrm{div}\upsilon_{+}+\varrho_{+}{\dot{\Delta}}_{q}\mathrm{div}\upsilon_{+}),\\
[2mm]
\partial_{t} {\dot{\Delta}}_{q}\varrho_{-}+{\dot{\Delta}}_{q}\mathrm{div}\upsilon_{-}\\ \hspace{5mm}=-(\upsilon_{-}\cdot\nabla){\dot{\Delta}}_{q}\varrho_{-}+[\upsilon_{-},{\dot{\Delta}}_{q}]\cdot\nabla\varrho_{-}
-\frac {\gamma-1}{2}([{\dot{\Delta}}_{q},\varrho_{-}]\mathrm{div}\upsilon_{-}+\varrho_{-}{\dot{\Delta}}_{q}\mathrm{div}\upsilon_{-}),\\[2mm]
\partial_{t} {\dot{\Delta}}_{q}\upsilon_{+}+{\dot{\Delta}}_{q}\nabla\varrho_{+}+\frac {1}{\sqrt{\gamma}}{\dot{\Delta}}_{q}\upsilon_{+}+\frac {1}{\sqrt{\gamma}}{\dot{\Delta}}_{q}\tilde{E}+\dot{\Delta}_{q}\upsilon_{+}\times\bar{B}\\ \hspace{5mm}=
-(\upsilon_{+}\cdot\nabla){\dot{\Delta}}_{q}\upsilon_{+}+[\upsilon_{+},{\dot{\Delta}}_{q}]\cdot\nabla\upsilon_{+}
-\frac
{\gamma-1}{2}([{\dot{\Delta}}_{q},\varrho_{+}]\nabla\varrho_{+}+\varrho_{+}{\dot{\Delta}}_{q}\nabla\varrho_{+})
-{\dot{\Delta}}_{q}(\upsilon_{+}\times \tilde{B}), \\[2mm]
\partial_{t} {\dot{\Delta}}_{q}\upsilon_{-}+{\dot{\Delta}}_{q}\nabla\varrho_{-}+\frac {1}{\sqrt{\gamma}}{\dot{\Delta}}_{q}\upsilon_{-}
-\frac
{1}{\sqrt{\gamma}}{\dot{\Delta}}_{q}\tilde{E}-\dot{\Delta}_{q}\upsilon_{-}\times\bar{B}\\
\hspace{5mm} =
-(\upsilon_{-}\cdot\nabla){\dot{\Delta}}_{q}\upsilon_{-}+[\upsilon_{-},{\dot{\Delta}}_{q}]\cdot\nabla\upsilon_{-}
-\frac
{\gamma-1}{2}([{\dot{\Delta}}_{q},\varrho_{-}]\nabla\varrho_{-}+\varrho_{-}{\dot{\Delta}}_{q}\nabla\varrho_{-})
+{\dot{\Delta}}_{q}(\upsilon_{-}\times \tilde{B}), \\[2mm]
\partial_{t} {\dot{\Delta}}_{q}\tilde{E}-\frac {1}{\sqrt{\gamma}}{\dot{\Delta}}_{q}\nabla\times\tilde{B}\\ \hspace{5mm}=\frac {1}{\sqrt{\gamma}}{\dot{\Delta}}_{q}\upsilon_{+}
+\frac
{1}{\sqrt{\gamma}}{\dot{\Delta}}_{q}[(\Phi(\varrho_{+})+\varrho_{+})\upsilon_{+}]-\frac
{1}{\sqrt{\gamma}}{\dot{\Delta}}_{q}\upsilon_{-}
-\frac {1}{\sqrt{\gamma}}{\dot{\Delta}}_{q}[(\Phi(\varrho_{-})+\varrho_{-})\upsilon_{-}], \\[2mm]
\partial_{t} {\dot{\Delta}}_{q}\tilde{B}+\frac {1}{\sqrt{\gamma}}{\dot{\Delta}}_{q}\nabla\times\tilde{E}=0, \\
\end{array}
\right.\label{R-E23}
\end{equation}
where the commutator $[\cdot,\cdot]$ is defined by $[f,g]=fg-gf$.

Then multiplying the first two equations of (\ref{R-E23}) by
${\dot{\Delta}}_{q}\varrho_{+},\ {\dot{\Delta}}_{q}\varrho_{-}$, the
third one by ${\dot{\Delta}}_{q}\upsilon_{+}$, the fourth one by
${\dot{\Delta}}_{q}\upsilon_{-}$, respectively, and adding the
resulting equations together after integrating them over
$\mathbb{R}^{N}$, we get
\begin{eqnarray}
&&\frac{1}{2}\frac{d}{dt}\|{\dot{\Delta}}_{q}(\varrho_{+},\varrho_{-},\upsilon_{+},\upsilon_{-})\|^2_{L^2}
+\frac{1}{\sqrt{\gamma}}\|{\dot{\Delta}}_{q}(\upsilon_{+},\upsilon_{-})\|^2_{L^2}+\frac{1}{\sqrt{\gamma}}\langle{\dot{\Delta}}_{q}\tilde{E},{\dot{\Delta}}_{q}\upsilon_{+}\rangle
-\frac{1}{\sqrt{\gamma}}\langle{\dot{\Delta}}_{q}\tilde{E},{\dot{\Delta}}_{q}\upsilon_{-}\rangle\nonumber
\\&&\hspace{35mm}={\sum_{i=+,-}}I^{i}_{1}(t)+{\sum_{i=+,-}}I^{i}_{2}(t),\label{R-E24}
\end{eqnarray}
where we have used the facts
$(\dot{\Delta}_{q}\upsilon_{\pm}\times\bar{B})\cdot\dot{\Delta}_{q}\upsilon_{\pm}=0$.
The energy functions in the right-side of (\ref{R-E24}) are defined
by
\begin{eqnarray*}I^{+}_{1}(t)&:=&\frac{1}{2}\langle\mathrm{div}\upsilon_{+},(|{\dot{\Delta}}_{q}\varrho_{+}|^2+|{\dot{\Delta}}_{q}\upsilon_{+}|^2)\rangle
+\frac{\gamma-1}{2}\langle{\dot{\Delta}}_{q}\varrho_{+},\nabla\varrho_{+}\cdot{\dot{\Delta}}_{q}\upsilon_{+}\rangle
\\&&-\langle{\dot{\Delta}}_{q}(\upsilon_{+}\times\tilde{B}),{\dot{\Delta}}_{q}\upsilon_{+}\rangle,\end{eqnarray*}
\begin{eqnarray*}I^{-}_{1}(t)&:=&\frac{1}{2}\langle\mathrm{div}\upsilon_{-},(|{\dot{\Delta}}_{q}\varrho_{-}|^2+|{\dot{\Delta}}_{q}\upsilon_{-}|^2)\rangle
+\frac{\gamma-1}{2}\langle{\dot{\Delta}}_{q}\varrho_{-},\nabla\varrho_{-}\cdot{\dot{\Delta}}_{q}\upsilon_{-}\rangle
\\&&+\langle{\dot{\Delta}}_{q}(\upsilon_{-}\times\tilde{B}),{\dot{\Delta}}_{q}\upsilon_{-}\rangle,\end{eqnarray*}
and
\begin{eqnarray*}
I^{+}_{2}(t)&:=&\langle[\upsilon_{+},{\dot{\Delta}}_{q}]\cdot\nabla\varrho_{+},{\dot{\Delta}}_{q}\varrho_{+}\rangle+\langle[\upsilon_{+},{\dot{\Delta}}_{q}]\cdot\nabla\upsilon_{+},{\dot{\Delta}}_{q}\upsilon_{+}\rangle
\\&&-\frac{\gamma-1}{2}\langle[{\dot{\Delta}}_{q},\varrho_{+}]\mathrm{div}\upsilon_{+},{\dot{\Delta}}_{q}\varrho_{+}\rangle-\frac{\gamma-1}{2}\langle[{\dot{\Delta}}_{q},\varrho_{+}]\nabla\varrho_{+},{\dot{\Delta}}_{q}\upsilon_{+}\rangle,
\end{eqnarray*}
\begin{eqnarray*}
I^{-}_{2}(t)&:=&\langle[\upsilon_{-},{\dot{\Delta}}_{q}]\cdot\nabla\varrho_{-},{\dot{\Delta}}_{q}\varrho_{-}\rangle+\langle[\upsilon_{-},{\dot{\Delta}}_{q}]\cdot\nabla\upsilon_{-},{\dot{\Delta}}_{q}\upsilon_{-}\rangle
\\&&-\frac{\gamma-1}{2}\langle[{\dot{\Delta}}_{q},\varrho_{-}]\mathrm{div}\upsilon_{-},{\dot{\Delta}}_{q}\varrho_{-}\rangle-\frac{\gamma-1}{2}\langle[{\dot{\Delta}}_{q},\varrho_{-}]\nabla\varrho_{-},{\dot{\Delta}}_{q}\upsilon_{-}\rangle.
\end{eqnarray*}

On the other hand, multiplying the fifth equation of (\ref{R-E23})
by ${\dot{\Delta}}_{q}\tilde{E}$ and the sixth one by
${\dot{\Delta}}_{q}\tilde{B}$, and adding the resulting equations
together after integrating them over $\mathbb{R}^{N}$ implies
\begin{eqnarray}&&\frac{1}{2}\frac{d}{dt}\|{\dot{\Delta}}_{q}(\tilde{E},\tilde{B})\|^2_{L^2}
-\frac{1}{\sqrt{\gamma}}\langle{\dot{\Delta}}_{q}\tilde{E},{\dot{\Delta}}_{q}\upsilon_{+}\rangle+\frac{1}{\sqrt{\gamma}}\langle{\dot{\Delta}}_{q}\tilde{E},{\dot{\Delta}}_{q}\upsilon_{-}\rangle\nonumber
\\&&\hspace{20mm}={\sum_{i=+,-}}I^{i}_{3}(t),\label{R-E25}
\end{eqnarray}
where we used the vector formula $\nabla\cdot(\vec{f}\times
\vec{g})=(\nabla\times\vec{f})\cdot\vec{g}-(\nabla\times\vec{g})\cdot\vec{f}$
and $I^{+}_{3}(t),I^{-}_{3}(t)$ are given by
$$
I^{+}_{3}(t):=\frac{1}{\sqrt{\gamma}}\langle{\dot{\Delta}}_{q}[(\Phi(\varrho_{+})+\varrho_{+})\upsilon_{+}],{\dot{\Delta}}_{q}\tilde{E}\rangle,
$$
$$
I^{-}_{3}(t):=-\frac{1}{\sqrt{\gamma}}\langle{\dot{\Delta}}_{q}[(\Phi(\varrho_{-})+\varrho_{-})\upsilon_{-}],{\dot{\Delta}}_{q}\tilde{E}\rangle.
$$

In what follows, we begin to bound these nonlinear terms. Firstly,
with the aid of Cauchy-Schwartz inequality, we have
\begin{eqnarray}
&&\int^T_{0}|I^{+}_{1}(t)|dt\nonumber
\\&\lesssim&\|(\nabla\varrho_{+},\nabla\upsilon_{+})\|_{L_{T}^2(L^{\infty})}\Big(2^{-q}\|{\dot{\Delta}}_{q}\nabla\varrho_{+}\|_{L^2_{T}(L^2)}\|{\dot{\Delta}}_{q}\varrho_{+}\|_{L^\infty_{T}(L^2)}
\nonumber
\\&&+2^{-q}\|{\dot{\Delta}}_{q}\nabla\varrho_{+}\|_{L^2_{T}(L^2)}\|{\dot{\Delta}}_{q}\upsilon_{+}\|_{L^\infty_{T}(L^2)}+\|{\dot{\Delta}}_{q}\upsilon_{+}\|_{L^2_{T}(L^2)}\|{\dot{\Delta}}_{q}\upsilon_{+}\|_{L^\infty_{T}(L^2)}\Big)
\nonumber
\\&&+\|{\dot{\Delta}}_{q}(\upsilon_{+}\times\tilde{B})\|_{L^2_{T}(L^2)}\|{\dot{\Delta}}_{q}\upsilon_{+}\|_{L^2_{T}(L^2)},\label{R-E26}
\end{eqnarray}
furthermore, multiplying the factor $2^{2qs_{c}}$ on both sides of
(\ref{R-E26}) gives
\begin{eqnarray}
&&2^{2qs_{c}}\int^T_{0}|I^{+}_{1}(t)|dt\nonumber
\\&\lesssim&c_{q}^2\|(\varrho_{+},\upsilon_{+})\|_{\widetilde{L}^\infty_{T}(\dot{B}^{s_{c}}_{2,1})}
\Big(\|\nabla\varrho_{+}\|^2_{\widetilde{L}^{2}_{T}(\dot{B}^{s_{c}-1}_{2,1})}+\|\nabla\varrho_{+}\|_{\widetilde{L}^{2}_{T}(\dot{B}^{s_{c}-1}_{2,1})}
\|\upsilon_{+}\|_{\widetilde{L}^{2}_{T}(\dot{B}^{s_{c}}_{2,1})}\nonumber
\\&&+\|\upsilon_{+}\|^2_{\widetilde{L}^{2}_{T}(\dot{B}^{s_{c}}_{2,1})}\Big)
+c_{q}^2\|\upsilon_{+}\times\tilde{B}\|_{\widetilde{L}^{2}_{T}(\dot{B}^{s_{c}}_{2,1})}\|\upsilon_{+}\|_{\widetilde{L}^{2}_{T}(\dot{B}^{s_{c}}_{2,1})},\label{R-E27}
\end{eqnarray}
where we used the embedding properties in Lemma \ref{lem2.2} and
Remark \ref{rem2.1}. Here and below, $\{c_{q}\}$ denotes some
sequence which satisfies $\|(c_{q})\|_{ {l^{1}}}\leq1$ although each
$\{c_{q}\}$ is possibly different in (\ref{R-E27}). Similarly, we
have
\begin{eqnarray}
&&2^{2qs_{c}}\int^T_{0}|I^{-}_{1}(t)|dt\nonumber
\\&\lesssim&c_{q}^2\|(\varrho_{-},\upsilon_{-})\|_{\widetilde{L}^\infty_{T}(\dot{B}^{s_{c}}_{2,1})}
\Big(\|\nabla\varrho_{-}\|^2_{\widetilde{L}^\infty_{T}(\dot{B}^{s_{c}-1}_{2,1})}+\|\nabla\varrho_{-}\|_{\widetilde{L}^\infty_{T}(\dot{B}^{s_{c}-1}_{2,1})}
\|\upsilon_{-}\|_{\widetilde{L}^{2}_{T}(\dot{B}^{s_{c}}_{2,1})}\nonumber
\\&&+\|\upsilon_{-}\|^2_{\widetilde{L}^{2}_{T}(\dot{B}^{s_{c}}_{2,1})}\Big)
+c_{q}^2\|\upsilon_{-}\times\tilde{B}\|_{\widetilde{L}^{2}_{T}(\dot{B}^{s_{c}}_{2,1})}\|\upsilon_{-}\|_{\widetilde{L}^{2}_{T}(\dot{B}^{s_{c}}_{2,1})}
.\label{R-E28}
\end{eqnarray}
Secondly, we turn to estimate the commutators occurring in
$I^{\pm}_{2}(t)$. Indeed, we arrive at
\begin{eqnarray}
&&2^{2qs_{c}}\int^T_{0}|I^{+}_{2}(t)|dt\nonumber
\\&\lesssim&2^{qs_{c}}\Big(\|[\upsilon_{+},{\dot{\Delta}}_{q}]\cdot\nabla\varrho_{+}\|_{L^2_{T}(L^2)}
+\|[\varrho_{+},{\dot{\Delta}}_{q}]\mathrm{div}\upsilon_{+}\|_{L^2_{T}(L^2)}\Big)2^{q(s_{c}-1)}\|{\dot{\Delta}}_{q}\nabla\varrho_{+}\|_{L^2_{T}(L^2)}
\nonumber
\\&&+2^{qs_{c}}\Big(\|[\upsilon_{+},{\dot{\Delta}}_{q}]\cdot\nabla\upsilon_{+}\|_{L^2_{T}(L^2)}
+\|[\varrho_{+},{\dot{\Delta}}_{q}]\nabla\varrho_{+}\|_{L^2_{T}(L^2)}\Big)2^{qs_{c}}\|{\dot{\Delta}}_{q}\upsilon_{+}\|_{L^2_{T}(L^2)}.\label{R-E29}
\end{eqnarray}
Taking advantage of the estimates of commutator in Proposition
\ref{prop2.3}, we obtain
\begin{eqnarray}
&&2^{2qs_{c}}\int^T_{0}|I^{+}_{2}(t)|dt\nonumber
\\&\lesssim& c_{q}^2\Big(\|\upsilon_{+}\|_{\widetilde{L}^\infty_{T}(\dot{B}^{s_{c}}_{2,1})}\|\nabla\varrho_{+}\|_{\widetilde{L}^{2}_{T}(\dot{B}^{s_{c}-1}_{2,1})}
+\|\varrho_{+}\|_{\widetilde{L}^\infty_{T}(\dot{B}^{s_{c}}_{2,1})}\|\nabla\upsilon_{+}\|_{\widetilde{L}^{2}_{T}(\dot{B}^{s_{c}-1}_{2,1})}\Big)
\|\nabla\varrho_{+}\|_{\widetilde{L}^{2}_{T}(\dot{B}^{s_{c}-1}_{2,1})}\nonumber
\\&&+c_{q}^2\Big(\|\upsilon_{+}\|_{\widetilde{L}^\infty_{T}(\dot{B}^{s_{c}}_{2,1})}\|\nabla\upsilon_{+}\|_{\widetilde{L}^{2}_{T}(\dot{B}^{s_{c}-1}_{2,1})}
+\|\varrho_{+}\|_{\widetilde{L}^\infty_{T}(\dot{B}^{s_{c}}_{2,1})}\|\nabla\varrho_{+}\|_{\widetilde{L}^{2}_{T}(\dot{B}^{s_{c}-1}_{2,1})}\Big)
\|\upsilon_{+}\|_{\widetilde{L}^{2}_{T}(\dot{B}^{s_{c}}_{2,1})}
\nonumber
\\&\lesssim&c_{q}^2\|(\varrho_{+},\upsilon_{+})\|_{\widetilde{L}^\infty_{T}(\dot{B}^{s_{c}}_{2,1})}
\nonumber
\\&&\hspace{5mm}\times\Big(\|\nabla\varrho_{+}\|^2_{\widetilde{L}^{2}_{T}(\dot{B}^{s_{c}-1}_{2,1})}+\|\nabla\varrho_{+}\|_{\widetilde{L}^{2}_{T}(\dot{B}^{s_{c}-1}_{2,1})}
\|\upsilon_{+}\|_{\widetilde{L}^{2}_{T}(\dot{B}^{s_{c}}_{2,1})}+\|\upsilon_{+}\|^2_{\widetilde{L}^{2}_{T}(\dot{B}^{s_{c}}_{2,1})}\Big).\label{R-E30}
\end{eqnarray}
Similarly,
\begin{eqnarray}
&&2^{2qs_{c}}\int^T_{0}|I^{-}_{2}(t)|dt\nonumber
\\&\lesssim& c_{q}^2\|(\varrho_{-},\upsilon_{-})\|_{\widetilde{L}^\infty_{T}(\dot{B}^{s_{c}}_{2,1})}
\nonumber
\\&&\hspace{5mm}\times\Big(\|\nabla\varrho_{-}\|^2_{\widetilde{L}^{2}_{T}(\dot{B}^{s_{c}-1}_{2,1})}+\|\nabla\varrho_{-}\|_{\widetilde{L}^{2}_{T}(\dot{B}^{s_{c}-1}_{2,1})}
\|\upsilon_{-}\|_{\widetilde{L}^{2}_{T}(\dot{B}^{s_{c}}_{2,1})}+\|\upsilon_{-}\|^2_{\widetilde{L}^{2}_{T}(\dot{B}^{s_{c}}_{2,1})}\Big).\label{R-E31}
\end{eqnarray}
Thirdly, the composition functions $I^{\pm}_{3}(t)$ can be estimated
as
\begin{eqnarray}
&&2^{2qs_{c}}\int^T_{0}|I^{+}_{3}(t)|dt\nonumber
\\&\lesssim&c_{q}^2\|\tilde{E}\|_{\widetilde{L}^\infty_{T}(\dot{B}^{s_{c}}_{2,1})}\Big(\|\Phi(\varrho_{+})\upsilon_{+}\|_{L^1_{T}(\dot{B}^{s_{c}}_{2,1})}
+\|\varrho_{+}\upsilon_{+}\|_{L^1_{T}(\dot{B}^{s_{c}}_{2,1})}\Big)\nonumber
\\&\lesssim&c_{q}^2\|\tilde{E}\|_{\widetilde{L}^\infty_{T}(\dot{B}^{s_{c}}_{2,1})}\Big(\|\Phi(\varrho_{+})\|_{\widetilde{L}^2_{T}(\dot{B}^{s_{c}}_{2,1})}+\|\varrho_{+}\|_{\widetilde{L}^2_{T}(\dot{B}^{s_{c}}_{2,1})}\Big)
\|\upsilon_{+}\|_{\widetilde{L}^2_{T}(\dot{B}^{s_{c}}_{2,1})}
\nonumber
\\&\lesssim&c_{q}^2\|\tilde{E}\|_{\widetilde{L}^\infty_{T}(\dot{B}^{s_{c}}_{2,1})}\|\nabla\varrho_{+}\|_{\widetilde{L}^2_{T}(\dot{B}^{s_{c}-1}_{2,1})}
\|\upsilon_{+}\|_{\widetilde{L}^2_{T}(\dot{B}^{s_{c}}_{2,1})},\label{R-E32}
\end{eqnarray}
where we used the corresponding homogeneous cases of Propositions
\ref{prop2.1}-\ref{prop2.2}. Similarly,
\begin{eqnarray}
2^{2qs_{c}}\int^T_{0}|I^{-}_{3}(t)|dt\lesssim
c_{q}^2\|\tilde{E}\|_{\widetilde{L}^\infty_{T}(\dot{B}^{s_{c}}_{2,1})}\|\nabla\varrho_{-}\|_{\widetilde{L}^2_{T}(\dot{B}^{s_{c}-1}_{2,1})}
\|\upsilon_{-}\|_{\widetilde{L}^2_{T}(\dot{B}^{s_{c}}_{2,1})}.\label{R-E33}
\end{eqnarray}

Together with the equalities (\ref{R-E24})-(\ref{R-E25}) and
inequalities (\ref{R-E27})-(\ref{R-E28}),
(\ref{R-E30})-(\ref{R-E33}), we conclude that there exists a
constant $\tilde{\mu}_{2}>0$ such that
\begin{eqnarray}
&&2^{2qs_{c}}\|{\dot{\Delta}}_{q}W(t)\|^2_{L^2}+\tilde{\mu}_{2}^22^{2qs_{c}}\|({\dot{\Delta}}_{q}\upsilon_{+},{\dot{\Delta}}_{q}\upsilon_{-})\|^2_{L^{2}_{t}(L^2)}
\nonumber
\\&\lesssim&2^{2qs_{c}}\|{\dot{\Delta}}_{q}W_{0}\|^2_{L^2}+\mbox{the right sides
of}\ \Big\{(\ref{R-E27})-(\ref{R-E28}),
(\ref{R-E30})-(\ref{R-E33})\Big\}.\label{R-E34}
\end{eqnarray}
Then it follows from the classical Young's inequality($\sqrt{fg}\leq
(f+g)/2,\ f,g\geq0$) that
\begin{eqnarray}
&&2^{qs_{c}}\|{\dot{\Delta}}_{q}W\|_{L^{\infty}_{T}(L^2)}+\tilde{\mu}_{2}2^{qs_{c}}\|({\dot{\Delta}}_{q}\upsilon_{+},{\dot{\Delta}}_{q}\upsilon_{-})\|_{L^{2}_{T}(L^2)}
\nonumber
\\&\lesssim&2^{qs_{c}}\|{\dot{\Delta}}_{q}W_{0}\|_{L^2}+c_{q}\sqrt{\|W\|_{\widetilde{L}_{T}^{\infty}(\dot{B}^{s_{c}}_{2,1})}}
\Big(\|\nabla\varrho_{\pm}\|_{\widetilde{L}_{T}^2(\dot{B}^{s_{c}-1}_{2,1})}+\|\upsilon_{\pm}\|_{\widetilde{L}_{T}^2(\dot{B}^{s_{c}}_{2,1})}\Big).
\label{R-E35}
\end{eqnarray}
Hence, summing up (\ref{R-E35}) on $q\in \mathbb{Z}$ gives
immediately
\begin{eqnarray}
&&\|W\|_{\widetilde{L}^{\infty}_{T}(\dot{B}^{s_{c}}_{2,1})}+\tilde{\mu}_{2}\|(\upsilon_{+},\upsilon_{-})\|_{\widetilde{L}^{2}_{T}(\dot{B}^{s_{c}}_{2,1})}
\nonumber
\\&\lesssim&\|W_{0}\|_{\dot{B}^{s_{c}}_{2,1}}+\sqrt{\|W\|_{\widetilde{L}_{T}^{\infty}(\dot{B}^{s_{c}}_{2,1})}}
\Big(\|(\nabla\varrho_{+},\nabla\varrho_{-})\|_{\widetilde{L}_{T}^2(\dot{B}^{s_{c}-1}_{2,1})}+\|(\upsilon_{+},\upsilon_{-})\|_{\widetilde{L}_{T}^2(\dot{B}^{s_{c}}_{2,1})}\Big).
\label{R-E36}
\end{eqnarray}\\

\noindent\textit{\underline{Step 2. The $L^2_{T}(L^2)$ estimates of
$(\upsilon_{+},\upsilon_{-})$}}.

It follows from (\ref{R-E6}) and usual energy methods, we can get
the equalities
\begin{eqnarray}
&&\frac{1}{2}\frac{d}{dt}\|(\varrho_{+},\varrho_{-},\upsilon_{+},\upsilon_{-})\|^2_{L^2}+\frac{1}{\sqrt{\gamma}}\|(\upsilon_{+},\upsilon_{-})\|^2_{L^2}
+\frac{1}{\sqrt{\gamma}}\langle\tilde{E},\upsilon_{+}\rangle-\frac{1}{\sqrt{\gamma}}\langle\tilde{E},\upsilon_{-}\rangle\nonumber
\\&=&\frac{\gamma-1}{2}\langle\varrho_{+},\upsilon_{+}\cdot\nabla\varrho_{+}\rangle+\frac{\gamma-1}{2}\langle\varrho_{-},\upsilon_{-}\cdot\nabla\varrho_{-}\rangle
-\langle\nabla\upsilon_{+},\upsilon_{+}^2\rangle-\langle\nabla\upsilon_{-},\upsilon_{-}^2\rangle\nonumber
\\&&-\langle\upsilon_{+}\times\tilde{B},\upsilon_{+}\rangle+\langle\upsilon_{-}\times\tilde{B},\upsilon_{-}\rangle, \label{R-E37}
\end{eqnarray}
and
\begin{eqnarray}
&&\frac{1}{2}\frac{d}{dt}\|(\tilde{E},\tilde{B})\|^2_{L^2}-\frac{1}{\sqrt{\gamma}}\langle\tilde{E},\upsilon_{+}\rangle+\frac{1}{\sqrt{\gamma}}\langle\tilde{E},\upsilon_{-}\rangle\nonumber
\\&=&\frac{1}{\sqrt{\gamma}}\langle[(\Phi(\varrho_{+})+\varrho_{+})\upsilon_{+}],\tilde{E}\rangle
-\frac{1}{\sqrt{\gamma}}\langle[(\Phi(\varrho_{-})+\varrho_{-})\upsilon_{-}],\tilde{E}\rangle.
\label{R-E38}
\end{eqnarray}
Combine (\ref{R-E37}) and (\ref{R-E38}) to get
\begin{eqnarray}
&&\frac{d}{dt}\|W\|^2_{L^2}+\frac{2}{\sqrt{\gamma}}\|(\upsilon_{+},\upsilon_{-})\|^2_{L^2}\nonumber
\\&\lesssim&\|(\varrho_{+},\varrho_{-})\|_{L^\infty}\Big(\|\upsilon_{+}\|_{L^2}\|\nabla\varrho_{+}\|_{L^2}
+\|\upsilon_{-}\|_{L^2}\|\nabla\varrho_{-}\|_{L^2}\Big)\nonumber
\\&&+\|(\nabla\upsilon_{+},\nabla\upsilon_{-},\tilde{B})\|_{L^\infty}\Big(\|\upsilon_{+}\|^2_{L^2}
+\|\upsilon_{-}\|^2_{L^2}\Big)\nonumber
\\&&+\|(\Phi(\varrho_{+}),\varrho_{+},\Phi(\varrho_{-}),\varrho_{-})\|_{L^\infty}\Big(\|\upsilon_{+}\|_{L^2}
+\|\upsilon_{-}\|_{L^2}\Big)\|\tilde{E}\|_{L^2}.\label{R-E39}
\end{eqnarray}
By integrating (\ref{R-E39}) with respect to $t\in[0,T]$, we obtain
\begin{eqnarray}
&&\|W\|_{L^\infty_{T}(L^2)}+(2\gamma^{-\frac{1}{2}})^{\frac{1}{2}}\|(\upsilon_{+},\upsilon_{-})\|_{L^2_{T}(L^2)}
\nonumber
\\&\lesssim&\|W_{0}\|_{L^2}+\sqrt{\|(\varrho_{+},\varrho_{-},\nabla\upsilon_{+},\nabla\upsilon_{-},\tilde{B},\Phi(\varrho_{+}),\Phi(\varrho_{-}))\|_{L^\infty_{T}(L^\infty)}}
\nonumber
\\&&\hspace{5mm}\times\Big(\|(\nabla\varrho_{+},\nabla\varrho_{-})\|_{L^2_{T}(L^2)}+\|(\upsilon_{+},\upsilon_{-},\tilde{E})\|_{L^2_{T}(L^2)}\Big)
\nonumber
\\&\lesssim&\|W_{0}\|_{L^2}+\sqrt{\|(\varrho_{\pm},\upsilon_{\pm},\tilde{B})\|_{\widetilde{L}_{T}^{\infty}(B^{s_{c}}_{2,1})}}\nonumber
\\&&\hspace{5mm}\times\Big(\|(\nabla\varrho_{+},\nabla\varrho_{-})\|_{L^2_{T}(L^2)}
+\|(\upsilon_{+},\upsilon_{-})\|_{L^2_{T}(L^2)}+\|\tilde{E}\|_{\widetilde{L}_{T}^{2}(B^{s_{c}-1}_{2,1})}\Big).
\label{R-E40}
\end{eqnarray}\\

\noindent\textit{\underline{Step 3. The
$\widetilde{L}^2_{T}(B^{s_{c}}_{2,1})$ estimates of
$(\upsilon_{+},\upsilon_{-})$}}.

Recently, Xu and Kawashima \cite{XK} obtained an elementary fact
that indicates the connection between the homogeneous
Chemin-Lerner's spaces and inhomogeneous Chemin-Lerner's spaces.
Precisely, it reads as follows: let $s>0, 1\leq\theta, p,
r\leq+\infty$. When $\theta\geq r$, it holds that
\begin{eqnarray}L^{\theta}_{T}(L^p)\cap\widetilde{L}^{\theta}_{T}(\dot{B}^{s}_{p,r})=\widetilde{L}^{\theta}_{T}(B^{s}_{p,r})\label{R-E41}\end{eqnarray}
for any $T>0$. Notice this fact and (1) in Lemma \ref{lem2.2}, the
inequality (\ref{R-E22}) directly follows from (\ref{R-E36}) and
(\ref{R-E40}) with
$\mu_{2}=\min(\tilde{\mu}_{2},(2\gamma^{-\frac{1}{2}})^{\frac{1}{2}})$.
\end{proof}
\begin{lem}\label{lem4.2}
If $W\in\widetilde{\mathcal{C}}_{T}(B^{s_{c}}_{2,1})\cap
\widetilde{\mathcal{C}}^1_{T}(B^{s_{c}-1}_{2,1})$ is a solution of
(\ref{R-E6})-(\ref{R-E7}) for any $T>0$, then the following estimate
holds:
\begin{eqnarray}
&&\|\varrho_{+}-\varrho_{-}\|_{\widetilde{L}^2_{T}(B^{s_{c}}_{2,1})}+\|(\nabla\varrho_{+},\nabla\varrho_{-})\|_{\widetilde{L}^2_{T}(B^{s_{c}-1}_{2,1})}\nonumber
\\&\lesssim&(\|(\varrho_{\pm},\upsilon_{\pm})\|_{\widetilde{L}^\infty_{T}(B^{s_{c}}_{2,1})}+\|(\varrho_{\pm0},\upsilon_{\pm0})\|_{B^{s_{c}}_{2,1}})
+\|(\upsilon_{+},\upsilon_{-})\|_{\widetilde{L}^2_{T}(B^{s_{c}}_{2,1})}\nonumber
\\&&+\|(\varrho_{\pm},\upsilon_{\pm},\tilde{B})\|_{\widetilde{L}^\infty_{T}(B^{s_{c}}_{2,1})}
\Big(\|(\nabla\varrho_{\pm},\varrho_{+}-\varrho_{-})\|_{\widetilde{L}^2_{T}(B^{s_{c}-1}_{2,1})}+\|(\upsilon_{+},\upsilon_{-})\|_{\widetilde{L}^2_{T}(B^{s_{c}}_{2,1})}\Big).
\label{R-E42}
\end{eqnarray}
\end{lem}
\begin{proof}
The proof is divided into two claims for clarity.\\

\noindent\textit{\underline{Claim 1.}} Under the assumptions of
Lemma \ref{lem4.2}, it holds that
\begin{eqnarray}
&&\|(\nabla\varrho_{+},\nabla\varrho_{-})\|_{\widetilde{L}^2_{T}(B^{s_{c}-1}_{2,1})}+\|\varrho_{+}-\varrho_{-}\|_{\widetilde{L}^2_{T}(B^{s_{c}-1}_{2,1})}\nonumber
\\&\lesssim&(\|(\varrho_{\pm},\upsilon_{\pm})\|_{\widetilde{L}^\infty_{T}(B^{s_{c}}_{2,1})}+\|(\varrho_{\pm0},\upsilon_{\pm0})\|_{B^{s_{c}}_{2,1}})
+\|(\upsilon_{+},\upsilon_{-})\|_{\widetilde{L}^2_{T}(B^{s_{c}}_{2,1})}\nonumber
\\&&+\|(\varrho_{\pm},\upsilon_{\pm},\tilde{B})\|_{\widetilde{L}^\infty_{T}(B^{s_{c}}_{2,1})}
\Big(\|(\nabla\varrho_{\pm},\varrho_{+}-\varrho_{-})\|_{\widetilde{L}^2_{T}(B^{s_{c}-1}_{2,1})}
+\|(\upsilon_{+},\upsilon_{-})\|_{\widetilde{L}^2_{T}(B^{s_{c}}_{2,1})}\Big).
\label{R-E43}
\end{eqnarray}
To do this, the first four equations of (\ref{R-E6}) can be
rewritten as
\begin{equation}
\partial_{t} \varrho_{+}+\mathrm{div}\upsilon_{+}=f^{+}_{1}\label{R-E44}
\end{equation}
\begin{equation}
\partial_{t} \varrho_{-}+\mathrm{div}\upsilon_{-}=f^{-}_{1}\label{R-E45}
\end{equation}
\begin{equation}
\partial_{t} \upsilon_{+}+\nabla\varrho_{+}+\frac{1}{\sqrt{\gamma}}\upsilon_{+}+\frac{1}{\sqrt{\gamma}}\tilde{E}=f^{+}_{2}\label{R-E46}
\end{equation}
\begin{equation}
\partial_{t} \upsilon_{-}+\nabla\varrho_{-}+\frac{1}{\sqrt{\gamma}}\upsilon_{-}-\frac{1}{\sqrt{\gamma}}\tilde{E}=f^{-}_{2}\label{R-E47}
\end{equation}
where
$$\left\{
\begin{array}{l}
f^{+}_{1}=-\upsilon_{+}\cdot\nabla\varrho_{+}-\frac{\gamma-1}{2}\varrho_{+}\mathrm{div}\upsilon_{+},\\
f^{-}_{1}=-\upsilon_{-}\cdot\nabla\varrho_{-}-\frac{\gamma-1}{2}\varrho_{-}\mathrm{div}\upsilon_{-},\\
f^{+}_{2}=-\upsilon_{+}\cdot\nabla\upsilon_{+}-\frac{\gamma-1}{2}\varrho_{+}\nabla\varrho_{+}-\upsilon_{+}\times(\tilde{B}+\bar{B}),\\
f^{-}_{2}=-\upsilon_{-}\cdot\nabla\upsilon_{-}-\frac{\gamma-1}{2}\varrho_{-}\nabla\varrho_{-}+\upsilon_{-}\times(\tilde{B}+\bar{B}).
\end{array}\right.$$

Applying the \textit{inhomogeneous} operator $\Delta_{q}(q\geq-1)$
to (\ref{R-E46}) and multiplying the resulting equality by
$\Delta_{q}\nabla\varrho_{+}$ implies
\begin{eqnarray}
&&\frac{d}{dt}\langle\Delta_{q}\upsilon_{+},\Delta_{q}\nabla\varrho_{+}\rangle+\|\Delta_{q}\nabla\varrho_{+}\|^2_{L^2}+\frac{1}{\gamma}\|\Delta_{q}\varrho_{+}\|^2_{L^2}
-\frac{1}{\gamma}\langle\Delta_{q}\varrho_{-},\Delta_{q}\varrho_{+}\rangle\nonumber
\\&=&-\frac{1}{\gamma}\langle\Delta_{q}\Phi(\varrho_{+}),\Delta_{q}\varrho_{+}\rangle+\frac{1}{\gamma}\langle\Delta_{q}\Phi(\varrho_{-}),\Delta_{q}\varrho_{+}\rangle
+\langle\Delta_{q}f^{+}_{2},\Delta_{q}\nabla\varrho_{+}\rangle\nonumber
\\&&-\frac{1}{\sqrt{\gamma}}\langle\Delta_{q}\upsilon_{+},\Delta_{q}\nabla\varrho_{+}\rangle+\|\Delta_{q}\mathrm{div}\upsilon_{+}\|^2_{L^2}
-\langle\Delta_{q}f^{+}_{1},\Delta_{q}\mathrm{div}\upsilon_{+}\rangle,\label{R-E48}
\end{eqnarray}
where we used the last equation of (\ref{R-E6}) and (\ref{R-E44}).

In a similar way as above, we have
\begin{eqnarray}
&&\frac{d}{dt}\langle\Delta_{q}\upsilon_{-},\Delta_{q}\nabla\varrho_{-}\rangle+\|\Delta_{q}\nabla\varrho_{-}\|^2_{L^2}+\frac{1}{\gamma}\|\Delta_{q}\varrho_{-}\|^2_{L^2}
-\frac{1}{\gamma}\langle\Delta_{q}\varrho_{+},\Delta_{q}\varrho_{-}\rangle\nonumber
\\&=&-\frac{1}{\gamma}\langle\Delta_{q}\Phi(\varrho_{-}),\Delta_{q}\varrho_{-}\rangle+\frac{1}{\gamma}\langle\Delta_{q}\Phi(\varrho_{+}),\Delta_{q}\varrho_{-}\rangle
+\langle\Delta_{q}f^{-}_{2},\Delta_{q}\nabla\varrho_{-}\rangle\nonumber
\\&&-\frac{1}{\sqrt{\gamma}}\langle\Delta_{q}\upsilon_{-},\Delta_{q}\nabla\varrho_{-}\rangle+\|\Delta_{q}\mathrm{div}\upsilon_{-}\|^2_{L^2}
-\langle\Delta_{q}f^{-}_{1},\Delta_{q}\mathrm{div}\upsilon_{-}\rangle.\label{R-E49}\end{eqnarray}
Furthermore, we add (\ref{R-E48}) to (\ref{R-E49}) to get
\begin{eqnarray}
&&\frac{d}{dt}(\langle\Delta_{q}\upsilon_{+},\Delta_{q}\nabla\varrho_{+}\rangle+\langle\Delta_{q}\upsilon_{-},\Delta_{q}\nabla\varrho_{-}\rangle)
+\|(\Delta_{q}\nabla\varrho_{+},\Delta_{q}\nabla\varrho_{-})\|^2_{L^2}+\frac{1}{\gamma}\|\Delta_{q}\varrho_{+}-\Delta_{q}\varrho_{-}\|^2_{L^2}\nonumber
\\&=&\|\Delta_{q}\mathrm{div}\upsilon_{+}\|^2_{L^2}+\|\Delta_{q}\mathrm{div}\upsilon_{-}\|^2_{L^2}
+\frac{1}{\gamma}\langle\Delta_{q}\Phi(\varrho_{+}),\Delta_{q}\varrho_{-}-\Delta_{q}\varrho_{+}\rangle+\langle\Delta_{q}f^{+}_{2},\Delta_{q}\nabla\varrho_{+}\rangle\nonumber
\\&&+\langle\Delta_{q}f^{-}_{2},\Delta_{q}\nabla\varrho_{-}\rangle+\frac{1}{\gamma}\langle\Delta_{q}\Phi(\varrho_{-}),\Delta_{q}\varrho_{+}-\Delta_{q}\varrho_{-}\rangle
-\frac{1}{\sqrt{\gamma}}\langle\Delta_{q}\upsilon_{+},\Delta_{q}\nabla\varrho_{+}\rangle\nonumber
\\&&-\frac{1}{\sqrt{\gamma}}\langle\Delta_{q}\upsilon_{-},\Delta_{q}\nabla\varrho_{-}\rangle
-\langle\Delta_{q}f^{+}_{1},\Delta_{q}\mathrm{div}\upsilon_{+}\rangle-\langle\Delta_{q}f^{-}_{1},\Delta_{q}\mathrm{div}\upsilon_{-}\rangle.\label{R-E50}
\end{eqnarray}
From Young's inequality, there exists a constant $\mu_{3}>0$ such
that
\begin{eqnarray}
&&\frac{d}{dt}(\langle\Delta_{q}\upsilon_{+},\Delta_{q}\nabla\varrho_{+}\rangle+\langle\Delta_{q}\upsilon_{-},\Delta_{q}\nabla\varrho_{-}\rangle)
+\mu_{3}^2(\|(\Delta_{q}\nabla\varrho_{+},\Delta_{q}\nabla\varrho_{-})\|^2_{L^2}+\|\Delta_{q}\varrho_{+}-\Delta_{q}\varrho_{-}\|^2_{L^2})\nonumber
\\&\lesssim&\|(\Delta_{q}\mathrm{div}\upsilon_{+},\Delta_{q}\mathrm{div}\upsilon_{-})\|^2_{L^2}
+\|(\Delta_{q}\upsilon_{+},\Delta_{q}\upsilon_{-})\|^2_{L^2}+\|\Delta_{q}(\Phi(\varrho_{+})-\Phi(\varrho_{-}))\|^2_{L^2}\nonumber
\\&&+\Big(\|\Delta_{q}f^{+}_{1}\|^2_{L^2}+\|\Delta_{q}f^{+}_{2}\|^2_{L^2}\Big)+\Big(\|\Delta_{q}f^{-}_{1}\|^2_{L^2}+\|\Delta_{q}f^{-}_{2}\|^2_{L^2}\Big).\label{R-E51}
\end{eqnarray}

By integrating (\ref{R-E51}) with respect to $t\in[0,T]$, and
multiplying the factor $2^{2q(s_{c}-1)}$ on both sides of the
resulting inequality, we obtain
\begin{eqnarray}
&&\mu_{3}2^{q(s_{c}-1)}(\|(\Delta_{q}\nabla\varrho_{+},\Delta_{q}\nabla\varrho_{-})\|_{L^2_{T}(L^2)}+\|\Delta_{q}\varrho_{+}-\Delta_{q}\varrho_{-}\|_{L^2_{T}(L^2)})\nonumber
\\&\lesssim&c_{q}(\|(\varrho_{\pm},\upsilon_{\pm})\|_{\widetilde{L}_{T}^{\infty}(L^2)}
+\|(\varrho_{\pm0},\upsilon_{\pm0})\|_{B^{s_{c}}_{2,1}})\nonumber
\\&&+c_{q}\Big(\|(\upsilon_{+},\upsilon_{-})\|_{\widetilde{L}^2_{T}(B^{s_{c}}_{2,1})}
+\|\Phi(\varrho_{+})-\Phi(\varrho_{-})\|_{\widetilde{L}^2_{T}(B^{s_{c}-1}_{2,1})}\nonumber
\\&&+\|f^{+}_{1}\|_{\widetilde{L}^2_{T}(B^{s_{c}-1}_{2,1})}
+\|f^{+}_{2}\|_{\widetilde{L}^2_{T}(B^{s_{c}-1}_{2,1})}+\|f^{-}_{1}\|_{\widetilde{L}^2_{T}(B^{s_{c}-1}_{2,1})}+\|f^{-}_{2}\|_{\widetilde{L}^2_{T}(B^{s_{c}-1}_{2,1})}\Big).
\label{R-E52}
\end{eqnarray}
Now we estimate nonlinear terms in the right-side of (\ref{R-E52})
in turn. Firstly,
\begin{eqnarray}
\|f^{+}_{1}\|_{\widetilde{L}^2_{T}(B^{s_{c}-1}_{2,1})}&\lesssim&\|\upsilon_{+}\cdot\nabla\varrho_{+}\|_{\widetilde{L}^2_{T}(B^{s_{c}-1}_{2,1})}
+\|\varrho_{+}\mathrm{div}\upsilon_{+}\|_{\widetilde{L}^2_{T}(B^{s_{c}-1}_{2,1})}\nonumber\\&\lesssim&
\|\upsilon_{+}\|_{\widetilde{L}^\infty_{T}(B^{s_{c}}_{2,1})}\|\nabla\varrho_{+}\|_{\widetilde{L}^2_{T}(B^{s_{c}-1}_{2,1})}
+\|\varrho_{+}\|_{\widetilde{L}^\infty_{T}(B^{s_{c}}_{2,1})}\|\|\upsilon_{+}\|_{\widetilde{L}^2_{T}(B^{s_{c}}_{2,1})},\label{R-E53}
\end{eqnarray}
where we used Proposition \ref{prop2.1} and Remark \ref{rem2.1}.
Similarly, we have
\begin{eqnarray}
\|f^{-}_{1}\|_{\widetilde{L}^2_{T}(B^{s_{c}-1}_{2,1})}\lesssim\|\upsilon_{-}\|_{\widetilde{L}^\infty_{T}(B^{s_{c}}_{2,1})}\|\nabla\varrho_{-}\|_{\widetilde{L}^2_{T}(B^{s_{c}-1}_{2,1})}
+\|\varrho_{-}\|_{\widetilde{L}^\infty_{T}(B^{s_{c}}_{2,1})}\|\|\upsilon_{-}\|_{\widetilde{L}^2_{T}(B^{s_{c}}_{2,1})};\label{R-E54}
\end{eqnarray}
\begin{eqnarray}
\|f^{+}_{2}\|_{\widetilde{L}^2_{T}(B^{s_{c}-1}_{2,1})}&\lesssim&\|\upsilon_{+}\cdot\nabla\upsilon_{+}\|_{\widetilde{L}^2_{T}(B^{s_{c}-1}_{2,1})}
+\|\varrho_{+}\nabla\varrho_{+}\|_{\widetilde{L}^2_{T}(B^{s_{c}-1}_{2,1})}
+\|\upsilon_{+}\times(\tilde{B}+\bar{B})\|_{\widetilde{L}^2_{T}(B^{s_{c}-1}_{2,1})}\nonumber\\&\lesssim&
\|\upsilon_{+}\|_{\widetilde{L}^\infty_{T}(B^{s_{c}}_{2,1})}\|\upsilon_{+}\|_{\widetilde{L}^2_{T}(B^{s_{c}}_{2,1})}
+\|\varrho_{+}\|_{\widetilde{L}^\infty_{T}(B^{s_{c}}_{2,1})}\|\nabla\varrho_{+}\|_{\widetilde{L}^2_{T}(B^{s_{c}-1}_{2,1})}\nonumber\\&&+
(1+\|\tilde{B}\|_{\widetilde{L}^\infty_{T}(B^{s_{c}}_{2,1})})\|\upsilon_{+}\|_{\widetilde{L}^2_{T}(B^{s_{c}}_{2,1})};\label{R-E55}
\end{eqnarray}
\begin{eqnarray}
\|f^{-}_{2}\|_{\widetilde{L}^2_{T}(B^{s_{c}-1}_{2,1})}&\lesssim&
\|\upsilon_{-}\|_{\widetilde{L}^\infty_{T}(B^{s_{c}}_{2,1})}\|\upsilon_{-}\|_{\widetilde{L}^2_{T}(B^{s_{c}}_{2,1})}
+\|\varrho_{-}\|_{\widetilde{L}^\infty_{T}(B^{s_{c}}_{2,1})}\|\nabla\varrho_{-}\|_{\widetilde{L}^2_{T}(B^{s_{c}-1}_{2,1})}\nonumber\\&&+
(1+\|\tilde{B}\|_{\widetilde{L}^\infty_{T}(B^{s_{c}}_{2,1})})\|\upsilon_{-}\|_{\widetilde{L}^2_{T}(B^{s_{c}}_{2,1})}.\label{R-E56}
\end{eqnarray}

To estimate the continuity of compositions
$\Phi(\varrho_{+})-\Phi(\varrho_{-})$, we need the further estimates
rather than that in Proposition \ref{prop2.2}. Indeed, the
Proposition \ref{prop5.1} in Appendix \ref{sec:5} will be used,
which is a natural generalization about the corresponding stationary
case in \cite{BCD}. In addition, we recall that $\Phi(\rho)$ is a
smooth function on the domain $\{\rho|\frac {\gamma-1}{2}\rho+1>0\}$
satisfying $\Phi'(0)=0$. Hence, take $s=s_{c}-1, \theta=2,
\theta_1=\theta_4=2, \theta_2=\theta_3=\infty, p=2, r=1$ in
(\ref{R-E74}) to get
\begin{eqnarray}
&&\|\Phi(\varrho_{+})-\Phi(\varrho_{-})\|_{\widetilde{L}^2_{T}(B^{s_{c}-1}_{2,1})}\nonumber\\&\lesssim&
\|\varrho_{+}-\varrho_{-}\|_{L^2_{T}(L^\infty)}(\|\varrho_{+}\|_{\widetilde{L}^\infty_{T}(B^{s_{c}-1}_{2,1})}
+\|\varrho_{-}\|_{\widetilde{L}^\infty_{T}(B^{s_{c}-1}_{2,1})})\nonumber\\&&+
\|\varrho_{+}-\varrho_{-}\|_{\widetilde{L}^2_{T}(B^{s_{c}-1}_{2,1})}(\|\varrho_{+}\|_{L^\infty_{T}(L^\infty)}
+\|\varrho_{-}\|_{L^\infty_{T}(L^\infty)})
\nonumber\\&\lesssim&\|\varrho_{+}-\varrho_{-}\|_{\widetilde{L}^2_{T}(B^{s_{c}-1}_{2,1})}(\|\varrho_{+}\|_{\widetilde{L}^\infty_{T}(B^{s_{c}}_{2,1})}
+\|\varrho_{-}\|_{\widetilde{L}^\infty_{T}(B^{s_{c}}_{2,1})}).\label{R-E57}
\end{eqnarray}
Combining (\ref{R-E52})-(\ref{R-E57}), we are led to the estimate
\begin{eqnarray}
&&2^{q(s_{c}-1)}\Big(\|(\Delta_{q}\nabla\varrho_{+},\Delta_{q}\nabla\varrho_{-})\|_{L^2_{T}(L^2)}+\|\Delta_{q}\varrho_{+}-\Delta_{q}\varrho_{-}\|_{L^2_{T}(L^2)}\Big)\nonumber
\\&\lesssim&c_{q}(\|(\varrho_{\pm},\upsilon_{\pm})\|_{\widetilde{L}_{T}^{\infty}(L^2)}
+\|(\varrho_{\pm0},\upsilon_{\pm0})\|_{B^{s_{c}}_{2,1}})+c_{q}\|(\upsilon_{+},\upsilon_{-})\|_{\widetilde{L}^2_{T}(B^{s_{c}}_{2,1})}\nonumber
\\&&+c_{q}\|(\varrho_{\pm},\upsilon_{\pm},\tilde{B})\|_{\widetilde{L}^\infty_{T}(B^{s_{c}}_{2,1})}
\Big(\|(\nabla\varrho_{\pm},\varrho_{+}-\varrho_{-})\|_{\widetilde{L}^2_{T}(B^{s_{c}-1}_{2,1})}
+\|(\upsilon_{+},\upsilon_{-})\|_{\widetilde{L}^2_{T}(B^{s_{c}}_{2,1})}\Big).\label{R-E58}
\end{eqnarray}
Summing up (\ref{R-E58}) on $q\geq-1$, the inequality (\ref{R-E43})
is followed easily.

Next, we give the reason that
$\varrho_{+}-\varrho_{-}\in\widetilde{L}^2_{T}(B^{s_{c}}_{2,1})$.\\

 \noindent\textit{\underline{Claim 2.}}
If
$\varrho_{+}-\varrho_{-}\in\widetilde{L}^2_{T}(B^{s_{c}-1}_{2,1}),(\nabla\varrho_{+},\nabla\varrho_{-})\in\widetilde{L}^2_{T}(B^{s_{c}-1}_{2,1})$,
then
\begin{eqnarray*}
\varrho_{+}-\varrho_{-}\in\widetilde{L}^2_{T}(B^{s_{c}}_{2,1})
\end{eqnarray*}
and
\begin{eqnarray}
\|\varrho_{+}-\varrho_{-}\|_{\widetilde{L}^2_{T}(B^{s_{c}}_{2,1})}\lesssim(\|(\nabla\varrho_{+},\nabla\varrho_{-})\|_{\widetilde{L}^2_{T}(B^{s_{c}-1}_{2,1})}
+\|\varrho_{+}-\varrho_{-}\|_{\widetilde{L}^2_{T}(B^{s_{c}-1}_{2,1})}).\label{R-E59}
\end{eqnarray}

Indeed, by virtue of the triangle inequality, one has
\begin{eqnarray}\|\nabla(\varrho_{+}-\varrho_{-})\|_{\widetilde{L}^2_{T}(B^{s_{c}-1}_{2,1})}\lesssim\|(\nabla\varrho_{+},\nabla\varrho_{-})\|_{\widetilde{L}^2_{T}(B^{s_{c}-1}_{2,1})},\label{R-E60}\end{eqnarray}
which implies
$\nabla(\varrho_{+}-\varrho_{-})\in\widetilde{L}^2_{T}(B^{s_{c}-1}_{2,1})$,
furthermore, it follows from the fact (\ref{R-E41}) that
$\nabla(\varrho_{+}-\varrho_{-})\in\widetilde{L}^2_{T}(\dot{B}^{s_{c}-1}_{2,1})$.
According to Bernstein's inequality (Lemma \ref{lem2.1}), we obtain
\begin{eqnarray}\|\varrho_{+}-\varrho_{-}\|_{\widetilde{L}^2_{T}(\dot{B}^{s_{c}}_{2,1})}\lesssim\|\nabla(\varrho_{+}-\varrho_{-})\|_{\widetilde{L}^2_{T}(\dot{B}^{s_{c}-1}_{2,1})}
\lesssim\|\nabla(\varrho_{+}-\varrho_{-})\|_{\widetilde{L}^2_{T}(B^{s_{c}-1}_{2,1})}.\label{R-E61}\end{eqnarray}
On the other hand, thanks to the embeddings
$$\widetilde{L}^2_{T}(B^{s_{c}-1}_{2,1})\hookrightarrow L^2_{T}(B^{s_{c}-1}_{2,1})\hookrightarrow L^2_{T}(B^{s_{c}-1}_{2,2})\hookrightarrow L^2_{T}(L^2),$$
we deduce that
\begin{eqnarray}\|\varrho_{+}-\varrho_{-}\|_{L^2_{T}(L^2)}\lesssim\|\varrho_{+}-\varrho_{-}\|_{\widetilde{L}^2_{T}(B^{s_{c}-1}_{2,1})}.\label{R-E62}\end{eqnarray}
Then from the basic fact (\ref{R-E41}), the inequality (\ref{R-E59})
is achieved by (\ref{R-E61}) and (\ref{R-E62}) directly.

Finally, (\ref{R-E42}) follows from (\ref{R-E43}) and (\ref{R-E59}),
which completes the proof of Lemma \ref{lem4.2}.
\end{proof}

\begin{lem}\label{lem4.3}
If $W\in\widetilde{\mathcal{C}}_{T}(B^{s_{c}}_{2,1})\cap
\widetilde{\mathcal{C}}^1_{T}(B^{s_{c}-1}_{2,1})$ is a solution of
(\ref{R-E6})-(\ref{R-E7}) for any $T>0$, then the following estimate
holds:
\begin{eqnarray}
&&\|\tilde{E}\|_{\widetilde{L}^2_{T}(B^{s_{c}-1}_{2,1})}\nonumber
\\&\lesssim&\|(\upsilon_{\pm},\tilde{E})\|_{\widetilde{L}^{\infty}_{T}(B^{s_{c}}_{2,1})}
+\|(\upsilon_{\pm0},\tilde{E}_{0})\|_{B^{s_{c}}_{2,1}}+\Big\{\|(\upsilon_{+},\upsilon_{-})\|_{\widetilde{L}^2_{T}(B^{s_{c}}_{2,1})}\nonumber
\\&&+\|\nabla\tilde{B}\|_{\widetilde{L}^2_{T}(B^{s_{c}-2}_{2,1})}\
+\|(\varrho_{+}-\varrho_{-})\|_{\widetilde{L}^2_{T}(B^{s_{c}}_{2,1})}
+\sqrt{\|(\varrho_{\pm},\upsilon_{\pm},\tilde{B})\|_{\widetilde{L}^{\infty}_{T}(B^{s_{c}}_{2,1})}}\nonumber
\\&&\hspace{5mm}\times\Big(\|(\nabla\varrho_{+},\nabla\varrho_{-})\|_{\widetilde{L}^2_{T}(B^{s_{c}-1}_{2,1})}+\|(\upsilon_{+},\upsilon_{-})\|_{\widetilde{L}^2_{T}(B^{s_{c}}_{2,1})}
+\|\tilde{E}\|_{\widetilde{L}^2_{T}(B^{s_{c}-1}_{2,1})}\Big)\Big\},\label{R-E63}
\end{eqnarray}
\end{lem}
\begin{proof}
In fact, from (\ref{R-E46}) and (\ref{R-E47}), we have
\begin{equation}
\partial_{t} (\upsilon_{+}-\upsilon_{-})+(\nabla\varrho_{+}-\nabla\varrho_{-})
+\frac{2}{\sqrt{\gamma}}\tilde{E}=f^{+}_{2}-f^{-}_{2}-\frac{1}{\sqrt{\gamma}}(\upsilon_{+}-\upsilon_{-})\label{R-E64}
\end{equation}
Applying the inhomogeneous localization operator $\Delta_{q}$ to
(\ref{R-E64}), multiplying the resulting inequality by
$\Delta_{q}\tilde{E}$ and integrating it over $\mathbb{R}^{N}$ gives
\begin{eqnarray}
\frac{d}{dt}\langle\Delta_{q}(\upsilon_{+}-\upsilon_{-}),\Delta_{q}\tilde{E}\rangle+\frac{2}{\sqrt{\gamma}}\|\Delta_{q}\tilde{E}\|^2_{L^2}
=\sum_{i=1}^{2}J_{i}(t),\label{R-E65}
\end{eqnarray}
where
\begin{eqnarray*}J_{1}(t):&=&\langle\Delta_{q}(\upsilon_{+}-\upsilon_{-}),\Delta_{q}\partial_{t}\tilde{E}\rangle\nonumber
\\&=&\frac{1}{\sqrt{\gamma}}\|\Delta_{q}(\upsilon_{+}-\upsilon_{-})\|^2_{L^2}
+\frac{1}{\sqrt{\gamma}}\langle\Delta_{q}(\upsilon_{+}-\upsilon_{-}),\Delta_{q}(\nabla\times\tilde{B})\rangle\nonumber
\\&&+\frac{1}{\sqrt{\gamma}}\langle\Delta_{q}(\upsilon_{+}-\upsilon_{-}),\Delta_{q}[\Phi(\varrho_{+})\upsilon_{+}+\varrho_{+}\upsilon_{+}]\rangle
\nonumber
\\&&-\frac{1}{\sqrt{\gamma}}\langle\Delta_{q}(\upsilon_{+}-\upsilon_{-}),\Delta_{q}[\Phi(\varrho_{-})\upsilon_{-}+\varrho_{-}\upsilon_{-}]\rangle
\end{eqnarray*}
and
\begin{eqnarray*}
J_{2}(t):&=&-\frac{1}{\sqrt{\gamma}}\langle\Delta_{q}(\upsilon_{+}-\upsilon_{-}),\Delta_{q}\tilde{E}\rangle
-\langle\Delta_{q}(\nabla\varrho_{+}-\nabla\varrho_{-}),\Delta_{q}\tilde{E}\rangle\nonumber
\\&&+
\langle\Delta_{q}(f^{+}_{2}-f^{-}_{2}),\Delta_{q}\tilde{E}\rangle.
\end{eqnarray*}
Through the straight but a little tedious calculations, with the aid
of Propositions \ref{prop2.1}-\ref{prop2.2}, we can obtain
\begin{eqnarray}
&&2^{2q(s_{c}-1)}\int_{0}^{T}|J_{1}(t)|dt\nonumber
\\&\lesssim&c_{q}^2\|(\upsilon_{+},\upsilon_{-})\|^2_{\widetilde{L}^2_{T}(B^{s_{c}-1}_{2,1})}+c_{q}^2\|(\upsilon_{+},\upsilon_{-})\|_{\widetilde{L}^2_{T}(B^{s_{c}}_{2,1})}
\|\nabla\times\tilde{B}\|_{\widetilde{L}^2_{T}(B^{s_{c}-2}_{2,1})}\nonumber
\\&&+c_{q}^2\|(\varrho_{+},\varrho_{-})\|_{\widetilde{L}^\infty_{T}(B^{s_{c}-1}_{2,1})}\|(\upsilon_{+},\upsilon_{-})\|^2_{\widetilde{L}^2_{T}(B^{s_{c}-1}_{2,1})}\label{R-E66}
\end{eqnarray}
and
\begin{eqnarray}
&&2^{2q(s_{c}-1)}\int_{0}^{T}|J_{2}(t)|dt\nonumber
\\&\lesssim&c_{q}^2\Big(\|(\upsilon_{+},\upsilon_{-})\|_{\widetilde{L}^2_{T}(B^{s_{c}-1}_{2,1})}+\|(\varrho_{+}-\varrho_{-})\|_{\widetilde{L}^2_{T}(B^{s_{c}}_{2,1})}\Big)
\|\tilde{E}\|_{\widetilde{L}^2_{T}(B^{s_{c}-1}_{2,1})}
+c_{q}^2\|(\varrho_{\pm},\upsilon_{\pm},\tilde{B})\|_{\widetilde{L}^{\infty}_{T}(B^{s_{c}-1}_{2,1})}
\nonumber
\\&&\hspace{5mm}\times\Big(\|(\nabla\varrho_{+},\nabla\varrho_{-})\|_{\widetilde{L}^2_{T}(B^{s_{c}-1}_{2,1})}
+\|(\upsilon_{+},\upsilon_{-})\|_{\widetilde{L}^2_{T}(B^{s_{c}}_{2,1})}\Big)\|\tilde{E}\|_{\widetilde{L}^2_{T}(B^{s_{c}-1}_{2,1})}.
\label{R-E67}
\end{eqnarray}
Then, combining with (\ref{R-E65})-(\ref{R-E67}), we arrive at
\begin{eqnarray}
&&2^{2q(s_{c}-1)}\|\Delta_{q}\tilde{E}\|^2_{L^2_{T}(L^2)}\nonumber
\\&\lesssim&c_{q}^2\Big(\|(\upsilon_{+},\upsilon_{-},\tilde{E})\|^2_{\widetilde{L}^{\infty}_{T}(B^{s_{c}}_{2,1})}
+\|(\upsilon_{+0},\upsilon_{-0},\tilde{E}_{0})\|^2_{B^{s_{c}}_{2,1})}\Big)
+c_{q}^2\|(\upsilon_{+},\upsilon_{-})\|^2_{\widetilde{L}^2_{T}(B^{s_{c}-1}_{2,1})}\nonumber
\\&&+c_{q}^2\|(\upsilon_{+},\upsilon_{-})\|_{\widetilde{L}^2_{T}(B^{s_{c}}_{2,1})}
\|\nabla\times\tilde{B}\|_{\widetilde{L}^2_{T}(B^{s_{c}-2}_{2,1})}+c_{q}^2\Big(\|(\upsilon_{+},\upsilon_{-})\|_{\widetilde{L}^2_{T}(B^{s_{c}-1}_{2,1})}
\nonumber
\\&&+\|(\varrho_{+}-\varrho_{-})\|_{\widetilde{L}^2_{T}(B^{s_{c}}_{2,1})}\Big)
\|\tilde{E}\|_{\widetilde{L}^2_{T}(B^{s_{c}-1}_{2,1})}+c_{q}^2\|(\varrho_{+},\varrho_{-})\|_{\widetilde{L}^\infty_{T}(B^{s_{c}}_{2,1})}\|(\upsilon_{+},\upsilon_{-})\|^2_{\widetilde{L}^2_{T}(B^{s_{c}}_{2,1})}
\nonumber
\\&&+c_{q}^2\|(\varrho_{\pm},\upsilon_{\pm},\tilde{B})\|_{\widetilde{L}^{\infty}_{T}(B^{s_{c}}_{2,1})}
\Big(\|\nabla\varrho_{\pm}\|_{\widetilde{L}^2_{T}(B^{s_{c}-1}_{2,1})}
+\|\upsilon_{\pm}\|_{\widetilde{L}^2_{T}(B^{s_{c}}_{2,1})}\Big)\|\tilde{E}\|_{\widetilde{L}^2_{T}(B^{s_{c}-1}_{2,1})}.
\label{R-E68}
\end{eqnarray}
Then employing Young's inequality and summing up the resulting
inequality on $q\geq-1$ concludes the inequality (\ref{R-E40}).
\end{proof}

\begin{lem}\label{lem4.4}
If $W\in\widetilde{\mathcal{C}}_{T}(B^{s_{c}}_{2,1})\cap
\widetilde{\mathcal{C}}^1_{T}(B^{s_{c}-1}_{2,1})$ is a solution of
(\ref{R-E6})-(\ref{R-E7}) for any $T>0$ and, then the following
estimate holds:
\begin{eqnarray}
&&\|\nabla\tilde{B}\|_{\widetilde{L}^2_{T}(B^{s_{c}-2}_{2,1})}
\nonumber\\&\lesssim&\|(\tilde{E},\tilde{B})\|_{\widetilde{L}^\infty_{T}(B^{s_{c}}_{2,1})}
+\|(\tilde{E}_{0},\tilde{B}_{0})\|_{B^{s_{c}}_{2,1}}
+\Big\{\|(\upsilon_{+},\upsilon_{-})\|_{\widetilde{L}^2_{T}(B^{s_{c}}_{2,1})}\nonumber
\\&&+\|\tilde{E}\|_{\widetilde{L}^2_{T}(B^{s_{c}-1}_{2,1})}+\sqrt{\|\varrho_{\pm}\|_{\widetilde{L}^\infty_{T}(B^{s_{c}}_{2,1})}}
\Big(\|(\upsilon_{+},\upsilon_{-})\|_{\widetilde{L}^2_{T}(B^{s_{c}}_{2,1})}+\|\nabla\tilde{B}\|_{\widetilde{L}^2_{T}(B^{s_{c}-2}_{2,1})}\Big)\Big\},\label{R-E69}
\end{eqnarray}
\end{lem}
\begin{proof}
Indeed, multiplying both sides of the fifth equation of (\ref{R-E6})
by $-\Delta_{q}(\nabla\times\tilde{B})$, taking integrations in
$x\in\mathbb{R}^{N}$ and using integration by parts and replacing
$\partial_{t} \Delta_{q}\tilde{B}$ from the fourth equation of
(\ref{R-E6}), we arrive at
\begin{eqnarray}
&&-\frac{d}{dt}\langle\Delta_{q}(\nabla\times\tilde{E}),\Delta_{q}\tilde{B}\rangle+\frac{1}{\sqrt{\gamma}}\|\Delta_{q}(\nabla\times\tilde{B})\|^2_{L^2}
\nonumber\\&=&\frac{1}{\sqrt{\gamma}}\|\Delta_{q}(\nabla\times\tilde{E})\|^2_{L^2}
-\frac{1}{\sqrt{\gamma}}\langle\Delta_{q}\upsilon_{+},\Delta_{q}(\nabla\times\tilde{B})\rangle
+\frac{1}{\sqrt{\gamma}}\langle\Delta_{q}\upsilon_{-},\Delta_{q}(\nabla\times\tilde{B})\rangle\nonumber
\\&&-\frac{1}{\sqrt{\gamma}}\langle\Delta_{q}[\Phi(\varrho_{+})\upsilon_{+}+\varrho_{+}\upsilon_{+}]
-\Delta_{q}[\Phi(\varrho_{-})\upsilon_{-}+\varrho_{-}\upsilon_{-}],\Delta_{q}(\nabla\times\tilde{B})\rangle,\label{R-E70}
\end{eqnarray}
where we used the vector formula $\nabla\cdot(\vec{f}\times
\vec{g})=(\nabla\times\vec{f})\cdot\vec{g}-(\nabla\times\vec{g})\cdot\vec{f}$.

With the help of Cauchy-Schwartz inequality, we obtain
\begin{eqnarray}
&&-\frac{d}{dt}\langle\Delta_{q}(\nabla\times\tilde{E}),\Delta_{q}\tilde{B}\rangle+\frac{1}{\sqrt{\gamma}}\|\Delta_{q}(\nabla\times\tilde{B})\|^2_{L^2}\nonumber
\\&\leq&\frac{1}{\sqrt{\gamma}}\|\Delta_{q}(\nabla\times\tilde{E})\|^2_{L^2}
+\frac{1}{\sqrt{\gamma}}\|\Delta_{q}(\upsilon_{+},\upsilon_{-})\|_{L^2}\|\Delta_{q}(\nabla\times\tilde{B})\|_{L^2}\nonumber
\\&&+\frac{1}{\sqrt{\gamma}}\Big(\|\Delta_{q}(\Phi(\varrho_{+})\upsilon_{+})\|_{L^2}+\|\Delta_{q}(\varrho_{+}\upsilon_{+})\|_{L^2}\Big)\|\Delta_{q}(\nabla\times\tilde{B})\|_{L^2}\nonumber
\\&&+\frac{1}{\sqrt{\gamma}}\Big(\|\Delta_{q}(\Phi(\varrho_{-})\upsilon_{-})\|_{L^2}+\|\Delta_{q}(\varrho_{-}\upsilon_{-})\|_{L^2}\Big)\|\Delta_{q}(\nabla\times\tilde{B})\|_{L^2}.
\label{R-E71}
\end{eqnarray}
Note that the regularity of $\tilde{E}$ in Lemma \ref{lem4.4}, we
multiply (\ref{R-E71}) by the factor $2^{2q(s_{c}-2)}$ after
integrating (\ref{R-E71}) with respect to $t\in[0,T]$ to get
\begin{eqnarray}
&&2^{2q(s_{c}-2)}\|\Delta_{q}\nabla\tilde{B}\|^2_{L^2_{T}(L^2)}\nonumber
\\&\lesssim& c^2_{q}(\|\tilde{E}\|_{\widetilde{L}^\infty_{T}(B^{s_{c}-1}_{2,1})}
\|\tilde{B}\|_{\widetilde{L}^\infty_{T}(B^{s_{c}-2}_{2,1})}+\|\tilde{E}_{0}\|_{B^{s_{c}-1}_{2,1}}
\|\tilde{B}_{0}\|_{B^{s_{c}-2}_{2,1}})\nonumber
\\&&+c^2_{q}\Big\{\|\tilde{E}\|^2_{\widetilde{L}^2_{T}(B^{s_{c}-1}_{2,1})}+\|(\upsilon_{+},\upsilon_{-})\|_{\widetilde{L}^2_{T}(B^{s_{c}}_{2,1})}
\|\nabla\tilde{B}\|_{\widetilde{L}^2_{T}(B^{s_{c}-2}_{2,1})}\nonumber
\\&&+\|(\varrho_{+},\varrho_{-})\|_{\widetilde{L}^\infty_{T}(B^{s_{c}}_{2,1})}\|(\upsilon_{+},\upsilon_{-})\|_{\widetilde{L}^2_{T}(B^{s_{c}}_{2,1})}
\|\nabla\tilde{B}\|_{\widetilde{L}^2_{T}(B^{s_{c}-2}_{2,1})})\Big\},
\label{R-E72}
\end{eqnarray}
where we notice the incompressible property of $\tilde{B}$ and the
elementary relation $\|\nabla \vec{f}\|_{L^2}\approx\|\nabla\cdot
\vec{f}\|_{L^2}+\|\nabla\times \vec{f}\|_{L^2}.$

Furthermore, we apply Young's inequality to (\ref{R-E72}) and obtain
\begin{eqnarray}
&&2^{q(s_{c}-2)}\|\Delta_{q}\nabla\tilde{B}\|_{L^2_{T}(L^2)}\nonumber
\\&\lesssim&c_{q}(\|(\tilde{E},\tilde{B})\|_{\widetilde{L}^\infty_{T}(B^{s_{c}}_{2,1})}
+\|(\tilde{E}_{0},\tilde{B}_{0})\|_{B^{s_{c}}_{2,1}})
+c_{q}\Big\{\varepsilon\|\nabla\tilde{B}\|_{\widetilde{L}_{T}(B^{s_{c}-2}_{2,1})}+
\|(\upsilon_{+},\upsilon_{-})\|_{\widetilde{L}^2_{T}(B^{s_{c}}_{2,1})}\nonumber
\\&&+\|\tilde{E}\|_{\widetilde{L}^2_{T}(B^{s_{c}-1}_{2,1})}+\sqrt{\|(\varrho_{+},\varrho_{-})\|_{\widetilde{L}^\infty_{T}(B^{s_{c}}_{2,1})}}\Big(\|(\upsilon_{+},\upsilon_{-})\|_{\widetilde{L}^2_{T}(B^{s_{c}}_{2,1})}
+\|\nabla\tilde{B}\|_{\widetilde{L}_{T}(B^{s_{c}-2}_{2,1})}\Big)\Big\},\label{R-E73}
\end{eqnarray}
where we take $\varepsilon\leq1/2$.

Finally, after summing up (\ref{R-E45}) on $q\geq-1$, the  desired
inequality (\ref{R-E69}) is followed.
\end{proof}

With the help of Lemmas \ref{lem4.1}-\ref{lem4.4}, the inequality
(\ref{R-E20}) in Proposition \ref{prop4.1} follows, since we may
introduce some positive constants to eliminate the terms
$\|W\|_{\widetilde{L}^\infty_{T}(B^{s_{c}}_{2,1})}$,
$\|(\nabla\varrho_{+},\nabla\varrho_{+})\|_{\widetilde{L}^2_{T}(B^{s_{c}-1}_{2,1})}$,\
$\|(\upsilon_{+},\upsilon_{-})\|\|_{\widetilde{L}^2_{T}(B^{s_{c}}_{2,1})}$
and $\|\tilde{E}\|_{\widetilde{L}^2_{T}(B^{s_{c}-1}_{2,1})}$,
$\|\nabla\tilde{B}\|_{\widetilde{L}^2_{T}(B^{s_{c}-2}_{2,1})}$
arising in the right-hand sides of (\ref{R-E22}),(\ref{R-E42}),
(\ref{R-E63}) and (\ref{R-E69}). See \cite{X2} for similar details,
here, we omit them for brevity.

Having the Proposition \ref{prop3.1} and Proposition \ref{prop4.1},
Theorem \ref{thm1.2} (Global well-posedness) can be achieved by the
standard boot-strap
argument as in \cite{MN}. We give the outline of proof.\\

\noindent\textit{\underline{Proof of Theorem \ref{thm1.2}.}}

If the initial data satisfy
$\|(\varrho_{\pm0},\upsilon_{\pm0},\tilde{E}_{0},\tilde{B}_{0})\|_{B^{s_{c}}_{2,1}}\leq\frac{\delta_1}{2}$,
by Proposition \ref{prop3.1}, then we can determine a time
$T_{1}>0(T_1\leq T_{0})$ such that the local solutions of
(\ref{R-E6})-(\ref{R-E7}) exist in
$\widetilde{\mathcal{C}}_{T_{1}}(B^{s_{c}}_{2,1})$ satisfying
$\|(\varrho_{\pm},\upsilon_{\pm},\tilde{E},\tilde{B})\|_{\widetilde{L}^\infty_{T_{1}}(B^{s_{c}}_{2,1})}\leq
\delta_{1}$. Therefore from Proposition \ref{prop4.1}, the solutions
satisfy the \textit{a priori} estimate
$\|(\varrho_{\pm},\upsilon_{\pm},\tilde{E},\tilde{B})\|_{\widetilde{L}^\infty_{T_{1}}(B^{s_{c}}_{2,1})}\leq
C_{1}\|(\varrho_{\pm0},\upsilon_{\pm0},\tilde{E}_{0},\tilde{B}_{0})\|_{B^{s_{c}}_{2,1}}
\leq\frac{\delta_1}{2}$, provided
$\|(\varrho_{\pm0},\upsilon_{\pm0},\tilde{E}_{0},\tilde{B}_{0})\|_{B^{s_{c}}_{2,1}}\leq\frac{\delta_{1}}{2C_{1}}$.
So by Proposition \ref{prop3.1} again, the system
(\ref{R-E6})-(\ref{R-E7}) for $t\geq T_{1}$ with the initial data
$(\varrho_{\pm},\upsilon_{\pm},\tilde{E},\tilde{B})(T_{1})$ has a
unique solution $(\varrho_{\pm},\upsilon_{\pm},\tilde{E},\tilde{B})$
satisfying
$\|(\varrho_{\pm},\upsilon_{\pm},\tilde{E},\tilde{B})\|_{\widetilde{L}^\infty_{(T_{1},2T_{1})}(B^{s_{c}}_{2,1})}\leq\delta_{1}$,
furthermore,
$\|(\varrho_{\pm},\upsilon_{\pm},\tilde{E},\tilde{B})\|_{\widetilde{L}^\infty_{2T_{1}}(B^{s_{c}}_{2,1})}\leq\delta_{1}$.
Then by Proposition \ref{prop4.1} we have
$\|(\varrho_{\pm},\upsilon_{\pm},\tilde{E},\tilde{B})\|_{\widetilde{L}^\infty_{2T_{1}}(B^{s_{c}}_{2,1})}
\leq
C_{1}\|(\varrho_{\pm0},\upsilon_{\pm0},\tilde{E}_{0},\tilde{B}_{0})\|_{B^{s_{c}}_{2,1}}\leq\frac{\delta_1}{2}$.
Thus we can continuous the same process for $0\leq t\leq
nT_{1},n=3,4,\cdot\cdot\cdot,$ and finally get a global solution
$(\varrho_{\pm},\upsilon_{\pm},\tilde{E},\tilde{B})$ satisfies
\begin{eqnarray}&&\|(\varrho_{\pm},\upsilon_{\pm},\tilde{E},\tilde{B})\|_{\widetilde{L}^\infty(B^{s_{c}}_{2,1})}
\nonumber\\&&+\mu_{1}\Big\{\|(\varrho_{+}-\varrho_{-},\upsilon_{\pm})\|_{\widetilde{L}^2(B^{s_{c}}_{2,1})}
+\|(\nabla\varrho_{+},\nabla\varrho_{-},\tilde{E})\|_{\widetilde{L}^2(B^{s_{c}-1}_{2,1})}
+\|\nabla\tilde{B}\|_{\widetilde{L}^2(B^{s_{c}-2}_{2,1})}\Big\}
\nonumber\\&\leq&
C_{1}\|(\varrho_{\pm0},\upsilon_{\pm0},\tilde{E}_{0},\tilde{B}_{0})\|_{B^{s_{c}}_{2,1}}\leq\frac{\delta_{1}}{2}.\label{R-E21}
\end{eqnarray}
The choice of $\delta_{1}$ is sufficient to ensure that
$\frac{\gamma-1}{2}\varrho_{\pm}+1>0$. Taking
$\delta_{0}=\min(\frac{\delta_{1}}{2},\frac{\delta_{1}}{2C_{1}})$,
then it follows from Remark \ref{rem3.1} and the embedding
properties (Lemma \ref{lem2.2}) that $(n_{\pm},u_{\pm}, E,B)\in
\mathcal{C}^{1}([0,\infty)\times\mathbb{R}^{N})$ is a unique
classical solution of (\ref{R-E1})-(\ref{R-E2}) in the whole space.

For the periodic case, it suffice to prove the inequality
(\ref{R-E444}). Recall that the definition of mean value $\bar{f}$
in Theorem \ref{thm1.2}, we set $\bar{n}_{\pm0}=1$. Using the
density equations in (\ref{R-E1}), we see that $\bar{n}_{\pm}$ are
conservative quantities for all time $t>0$, so $\bar{n}_{\pm}(t)=1$.
From Poinc\'{a}re inequality (see, \textit{e.g.}, \cite{E}), we have
\begin{eqnarray}
\|n_{\pm}-1\|_{L^2_{T}(L^2(\mathbb{T}^N))}\lesssim\|\nabla
n_{\pm}\|_{L^2_{T}(L^2(\mathbb{T}^N))}\lesssim\|\nabla
n_{\pm}\|_{\widetilde{L}^2(B^{s_{c}-1}_{2,1}(\mathbb{T}^N))}.
\label{R-E74}
\end{eqnarray}
On the other hand, the Bernstein inequality (Lemma \ref{lem2.1})
implies
\begin{eqnarray}
\|n_{\pm}-1\|_{\widetilde{L}^2(\dot{B}^{s_{c}}_{2,1}(\mathbb{T}^N))}\lesssim\|\nabla
n_{\pm}\|_{\widetilde{L}^2(\dot{B}^{s_{c}-1}_{2,1}(\mathbb{T}^N))}.
\label{R-E75}
\end{eqnarray}
Apply the basic fact (\ref{R-E41}) again and get
\begin{eqnarray}
\|n_{\pm}-1\|_{\widetilde{L}^2(B^{s_{c}}_{2,1}(\mathbb{T}^N))}\lesssim\|\nabla
n_{\pm}\|_{\widetilde{L}^2(B^{s_{c}-1}_{2,1}(\mathbb{T}^N))}.\label{R-E76}
\end{eqnarray}
Hence, (\ref{R-E444}) follows from (\ref{R-E76}) and (\ref{R-E4})
readily. This completes Theorem \ref{thm1.2} eventually. $\square$

\section{Appendix}\setcounter{equation}{0} \label{sec:5}
In the last section, we present a remark on the continuity for
compositions in Chemin-Lerner spaces. The corresponding stationary
cases have been shown in \cite{BCD} (see Corollary 2.66, P.97 and
Corollary 2.91, P.105). Precisely, we have
\begin{prop}\label{prop5.1}
Let $s>0$, $1\leq p, r, \theta, \theta_1, \theta_2, \theta_3,
\theta_4\leq \infty$, $F''\in W^{[s]+1,\infty}_{loc}(I;\mathbb{R})$
with $F'(0)=0$ and $T\in (0,\infty]$.  Then
\begin{eqnarray}&&\|F(f)-F(g)\|_{\widetilde{L}^{\theta}_{T}(B^{s}_{p,r})}\nonumber\\&\lesssim&
(1+\|f\|_{L^{\infty}_{T}(L^{\infty})}+\|g\|_{L^{\infty}_{T}(L^{\infty})})^{[s]+1}\|F''\|_{W^{[s]+1,\infty}}\Big(\|f-g\|_{L^{\theta_{1}}_{T}(L^\infty)}
\nonumber\\&&\times\sup_{\kappa\in[0,1]}\|g+\kappa(f-g)\|_{\widetilde{L}^{\theta_{2}}_{T}(B^{s}_{p,r})}+\|f-g\|_{\widetilde{L}^{\theta_{4}}_{T}(B^{s}_{p,r})}
\sup_{\kappa\in[0,1]}\|g+\kappa(f-g)\|_{L^{\theta_{3}}_{T}(L^\infty)}\Big),
\label{R-E77}\end{eqnarray} where
$$\frac{1}{\theta}=\frac{1}{\theta_1}+\frac{1}{\theta_2}=\frac{1}{\theta_3}+\frac{1}{\theta_4}.$$
\end{prop}

\begin{proof}
Following from their suggestions in \cite{BCD}, we give the natural
generalization. Note that the classical equality
\begin{eqnarray}F(f)-F(g)=(f-g)\int_{0}^{1}F'(g+\kappa(f-g))d\kappa,\label{R-E78}\end{eqnarray} it follows
from Proposition \ref{prop2.1} and  \ref{prop2.2} that
\begin{eqnarray}
&&\|F(f)-F(g)\|_{\widetilde{L}^{\theta}_{T}(B^{s}_{p,r})}\nonumber\\&\lesssim&
\|f-g\|_{L^{\theta_{1}}_{T}(L^\infty)}\Big\|\int_{0}^{1}F'(g+\kappa(f-g))d\kappa\Big\|_{\widetilde{L}^{\theta_{2}}_{T}(B^{s}_{p,r})}
\nonumber\\&&+\Big\|\int_{0}^{1}F'(g+\kappa(f-g))d\kappa\Big\|_{L^{\theta_{3}}_{T}(L^\infty)}\|f-g\|_{\widetilde{L}^{\theta_{4}}_{T}(B^{s}_{p,r})},\label{R-E79}
\end{eqnarray}
where
\begin{eqnarray}&&\Big\|\int_{0}^{1}F'(g+\kappa(f-g))d\kappa\Big\|_{\widetilde{L}^{\theta_{2}}_{T}(B^{s}_{p,r})}\nonumber\\&\leq&\sup_{\kappa\in[0,1]}
\|F'(g+\kappa(f-g))\|_{\widetilde{L}^{\theta_{2}}_{T}(B^{s}_{p,r})}\nonumber\\&\lesssim&
\sup_{\kappa\in[0,1]}\Big((1+\|g+\kappa(f-g)\|_{L^{\infty}_{T}(L^\infty)})^{[s]+1}\|F''\|_{W^{[s]+1,\infty}}\|g+\kappa(f-g)\|_{\widetilde{L}^{\theta_{2}}_{T}(B^{s}_{p,r})}\Big)
\nonumber\\&\lesssim&(1+\|f\|_{L^{\infty}_{T}(L^\infty)}+\|g\|_{L^{\infty}_{T}(L^\infty)})^{[s]+1}\|F''\|_{W^{[s]+1,\infty}}
\sup_{\kappa\in[0,1]}\|g+\kappa(f-g)\|_{\widetilde{L}^{\theta_{2}}_{T}(B^{s}_{p,r})},\label{R-E80}
\end{eqnarray}
and
\begin{eqnarray}
&&\Big\|\int_{0}^{1}F'(g+\kappa(f-g))d\kappa\Big\|_{L^{\theta_{3}}_{T}(L^\infty)}\nonumber\\&\leq&
\sup_{\kappa\in[0,1]}
\|F'(g+\kappa(f-g))\|_{L^{\theta_{3}}_{T}(L^\infty)}\nonumber\\&\lesssim&\|F''\|_{L^\infty}\sup_{\kappa\in[0,1]}\|g+\kappa(f-g)\|_{L^{\theta_{3}}_{T}(L^\infty)}.\label{R-E81}
\end{eqnarray}
Therefore, (\ref{R-E77}) follows from (\ref{R-E79}), (\ref{R-E80})
and (\ref{R-E81}) readily.
\end{proof}

Similarly, let us also mention the case in homogeneous Chemin-Lerner
spaces.

\begin{prop}\label{prop5.2}
Let $s>0$, $1\leq p, r, \theta, \theta_1, \theta_2, \theta_3,
\theta_4\leq \infty$, $F''\in W^{[s]+1,\infty}_{loc}(I;\mathbb{R})$
with $F'(0)=0$ and $T\in (0,\infty]$. Besides, let $s<N/p$ or
$s=N/p$ and $r=1$. Then
\begin{eqnarray}&&\|F(f)-F(g)\|_{\widetilde{L}^{\theta}_{T}(\dot{B}^{s}_{p,r})}\nonumber\\&\lesssim&
(1+\|f\|_{L^{\infty}_{T}(L^{\infty})}+\|g\|_{L^{\infty}_{T}(L^{\infty})})^{[s]+1}\|F''\|_{W^{[s]+1,\infty}}\Big(\|f-g\|_{L^{\theta_{1}}_{T}(L^\infty)}
\nonumber\\&&\times\sup_{\kappa\in[0,1]}\|g+\kappa(f-g)\|_{\widetilde{L}^{\theta_{2}}_{T}(\dot{B}^{s}_{p,r})}+\|f-g\|_{\widetilde{L}^{\theta_{4}}_{T}(\dot{B}^{s}_{p,r})}
\sup_{\kappa\in[0,1]}\|g+\kappa(f-g)\|_{L^{\theta_{3}}_{T}(L^\infty)}\Big),
\label{R-E82}\end{eqnarray} where
$$\frac{1}{\theta}=\frac{1}{\theta_1}+\frac{1}{\theta_2}=\frac{1}{\theta_3}+\frac{1}{\theta_4}.$$
\end{prop}

\section*{Acknowledgments}
J.Xu is partially supported by the NSFC (11001127), China
Postdoctoral Science Foundation (20110490134) and Postdoctoral
Science Foundation of Jiangsu Province (1102057C). S. Kawashima is
partially supported by Grant-in-Aid for Scientific Research (A)
22244009. The first author (J.Xu) would like to thank Professor
R.J.Duan for his communication on the dissipative structure of
regularity-loss of Euler-Maxwell equations.

\end{document}